\documentclass[a4paper,11pt,reqno]{amsart}
\usepackage{amsmath,amsthm,latexsym,amscd,amssymb,hyperref}
\usepackage[all]{xy}
\usepackage{colortbl}
\usepackage{verbatim}
\usepackage{graphics,epsfig,color}
\usepackage{amsfonts,latexsym,enumerate,wasysym,stmaryrd}
\numberwithin{equation}{section}

\DeclareMathOperator{\Str}{Str}
\DeclareMathOperator{\tr}{tr}

\DeclareMathOperator{\trG}{\tr_{\Gamma}}
\DeclareMathOperator{\StrG}{\Str{}_{\Gamma}}
\DeclareMathOperator{\Tr}{Tr}

\DeclareMathOperator{\Cl}{Cl}
\DeclareMathOperator{\Op}{Op}

\DeclareMathOperator{\scal}{scal}

\DeclareMathOperator{\Diff}{Diff}
\DeclareMathOperator{\UDiff}{UDiff}
\DeclareMathOperator{\spec}{spec}

\DeclareMathOperator{\ch}{ch}
\DeclareMathOperator{\Ch}{Ch}

\newtheorem{proposition}{Proposition}[section]
\newtheorem{lemma}[proposition]{Lemma}
\newtheorem{theorem}[proposition]{Theorem}
\newtheorem{cor}[proposition]{Corollary}
\newtheorem{remark}[proposition]{Remark}

\newtheorem{conjecture}[equation]{Conjecture}
\newtheorem{remnot}[proposition]{Remarks and notation}
\theoremstyle{definition}
\newtheorem{definition}[proposition]{Definition}

\begin{document}

\newenvironment{com}%
{\par \vspace{0cm}\ \\ \noindent \%\color{blue}\% \begin{tiny}}{\par\end{tiny}\%\color{black}\%\vspace{0cm}}
\newcommand{\foot}[1]{\footnote{\begin{normalsize}#1\end{normalsize}}}
\def\dim{\mathop{\rm dim}}
\def\Re{\mathop{\rm Re}}
\def\Im{\mathop{\rm Im}}
\def\I{\mathop{\rm I}}
\def\Id{\mathop{\rm Id}}
\def\grad{\mathop{\rm grad}}
\def\vol{\mathop{\rm vol}}
\def\SU{\mathop{\rm SU}}
\def\SO{\mathop{\rm SO}}
\def\Aut{\mathop{\rm Aut}}
\def\End{\mathop{\rm End}}
\def\GL{\mathop{\rm GL}}
\def\Cinf{\mathop{\mathcal C^{\infty}}}
\def\Ker{\mathop{\rm Ker}}
\def\Coker{\mathop{\rm Coker}}
\def\dom{\mathop{\rm Dom}}
\def\Hom{\mathop{\rm Hom}}
\def\Ch{\mathop{\rm Ch}}
\def\sign{\mathop{\rm sign}}
\def\SF{\mathop{\rm SF}}
\def\loc{\mathop{\rm loc}}
\def\AS{\mathop{\rm AS}}
\def\ev{\mathop{\rm ev}}
\def\id{\mathop{\rm id}}
\def\Cli{\mathbb{C}l(1)}

\def\ds{\displaystyle}
\def\pn{\pi_{n}}
\def\pnu{\pi_{n-1}}
\def\Fi{\Phi}
\def\phi{\varphi}
\def\de{\delta}\def\e{\eta}\def\ep{\epsilon}\def\ro{\rho}\def\a{\alpha}\def\o{\omega}\def\O{\Omega}\def\b{\beta}\def\la{\lambda}\def\th{\theta}\def\s{\sigma}\def\t{\tau}\def\g{\gamma}
\def\D{\Delta}\def\G{\Gamma}
\def\R{\mathbin{\mathbb R}}
\def\Rn{\R^{n}}
\def\C{\mathbb{C}}
\def\Cm{\mathbb{C}^{m}}
\def\Cn{\mathbb{C}^{n}}
\def\w{{\mathchoice{\,{\scriptstyle\wedge}\,}{{\scriptstyle\wedge}}
{{\scriptscriptstyle\wedge}}{{\scriptscriptstyle\wedge}}}}
\def\cA{{\cal A}}\def\cL{{\cal L}}
\def\cO{{\cal O}}\def\cT{{\cal T}}\def\cU{{\cal U}}
\def\cD{{\cal D}}\def\cF{{\cal F}}\def\cP{{\cal P}}\def\cH{{\cal H}}\def\cL{{\cal L}}
\def\cB{{\cal B}}


\newcommand{\n}[1]{\left\| #1\right\|}


\def\Mt{\tilde{M}}
\def\Zt{\tilde{Z}}
\def\Nt{\tilde{N}}
\def\Et{\tilde{E}}
\def\Vt{\tilde{V}}
\def\Xt{\tilde{X}}
\def\N{\mathcal{V}}
\def\Cd{\mathbb{C}^{d}}
\def\Dt{\tilde{D}}
\def\Bt{\tilde{\mathbb{B}}}
\def\Ct{\tilde{C}_t}

\def\Cl{\mathop{\rm Cl}}

\newcommand\Di{D\kern-6.5pt/}
\newcommand\cDi{\mathcal{D}\kern-6.5pt/}
\def\cA{{\cal A}}\def\cL{{\cal L}}
\def\cO{{\cal O}}\def\cT{{\cal T}}\def\cU{{\cal U}}
\def\cD{{\cal D}}\def\cF{{\cal F}}\def\cP{{\cal P}}\def\cH{{\cal H}}\def\cL{{\cal L}}
\def\cB{{\cal B}}
\def\N{\mathcal{V}}

\def\T{\mathcal{T}}


\title[$L^2$-rho form for normal coverings of fibre bundles]{$\mathbf{L^2}$-rho form for normal coverings of fibre bundles}


\begin{abstract}
\noindent 
We define the secondary invariants $L^2$-eta and -rho forms for families of generalized Dirac operators on normal coverings of fibre bundles. On the covering family we assume transversally smooth spectral projections and Novikov--Shubin invariants bigger than $3(\dim B+1)$ to treat the large time asymptotic for general operators. In the case of a bundle of spin manifolds, we study the $L^2$-rho class in relation to the space $\mathcal R^+(M/B)$ of positive scalar curvature vertical metrics.
\end{abstract}

\author{Sara Azzali}
\address{Mathematisches Institut \\Georg-August Universit\"at G\"ottingen}
\email{azzali@uni-math.gwdg.de}
\date{September 4, 2010}
\maketitle

\setlength{\parskip}{\smallskipamount}


\setlength{\parskip}{\medskipamount}

\section{Introduction}

Secondary invariants of Dirac operators are a distinctive issue of the heat equation approach to index theory. The \emph{eta invariant} of a Dirac operator first appeared as the boundary term in the Atiyah--Patodi--Singer index theorem \cite{APS1}: this spectral invariant, highly nonlocal and therefore unstable, became a major object of investigation, because of its subtle relation to geometry. With the introduction of superconnnections in index theory by Quillen and Bismut, it became possible to employ heat equation techniques in higher geometric situations, where the primary invariant, the index, is no longer a number, but a class in a $K$-theory group \cite{Q, Bi, Lo4}. This led to so called \emph{local index theorems}, which are refinements of the cohomological index theorems at the level of differential forms, and gave as new fundamental byproduct the \emph{eta forms}, coming from the transgression of the index class \cite{BC,BF, BC1}, which are the higher analogue of eta invariants \cite{MP, lphigh, Lo3}.  

\emph{Rho invariants} are differences (or, more generally, delocalized parts) of eta invariants, so they naturally possess stability properties when computed for geometrically relevant operators, mainly the \emph{spin Dirac operator} and the \emph{signature operator} \cite{APS2, Ke, PS1}. Furthermore, they can be employed to detect geometric structures: the Cheeger--Gromov $L^2$-rho invariant, for example, has major applications in distinguishing positive scalar curvature metrics on spin manifolds \cite{BG,PS2}, and can show the existence of infinitely many manifolds that are homotopy equivalent but not diffeomorphic to a fixed one \cite{CW}.

As secondary invariants always accompany primary ones, it is very natural to ask what are the \emph{$L^2$-eta} and \emph{$L^2$-rho forms} in the case of a families, and what are their properties. 

We consider the easiest $L^2$-setting one could think of, namely a \emph{normal covering of a fibre bundle}. This interesting model contains yet all the features and problems offered by the presence of continuos spectrum.
 Since the fibres of the covering family are noncompact, the large time asymptotic of the \emph{superconnection Chern character} is in general not converging to a differential form representative of the index class, and the same problem is reflected when trying to integrate on $[1,\infty)$ the transgression term involved in the definition of the $L^2$-eta form. 

The major result in this sense is by Heitsch and Lazarov, who gave the first families index theorem for foliations with Hausdorff graph \cite{HL}. They computed the large time limit of the superconnection Chern character as Haefliger form, assuming smooth spectral projections and Novikov--Shubin invariants bigger than $3$ times the codimension of the foliation. 
Their result implies an index theorem in Haefliger cohomology (not a local one, because they do not deal with the transgression term), which in particular applies to the easier $L^2$-setting under consideration. 

We use the techniques of Heitsch--Lazarov to investigate the integrability on $[1,\infty)$ of the transgression term, in order to define the $L^2$-eta form for families $\mathcal D$ of generalised Dirac operators on normal coverings of fibre bundles. 
Our main result, Theorem \ref{th-eta}, implies that the $L^2$-eta form  $\hat{\eta}_{(2)}(\mathcal D)$ is well defined as a continuos differential form on the base $B$ if the spectral projections of the family $\mathcal D$ are smooth, and the families Novikov--Shubin invariants  $\{\a_{K}\}_{K\subset B}$ are greater than $3(\dim B+1)$.  

We define then naturally the $L^2$-rho form $\hat{\rho}_{(2)}(\mathcal D)$ as the difference between the $L^2$-eta form for the covering family and the eta form of the family of compact manifolds.
When the fibre is odd dimensional, the zero degree term of $\hat{\rho}_{(2)}(\mathcal D)$ is the Cheeger--Gromov $L^2$-rho invariant of the induced covering of the fibre. We prove that the $L^2$-form is (weakly) closed when the fibres are odd dimensional (Prop. \ref{wclosed}).

The strong assumptions of Theorem \ref{th-eta} are required because we want to define $\hat{\eta}_{(2)}$ for a family of  generalised Dirac operators. In the particular case of de Rham and signature operators one can put weaker assumptions: this is showed by Gong--Rothenberg's result for the $L^2$-Bismut--Lott index theorem  (proved under positivity of the Novikov--Shubin invariants) \cite{GR}, and from results in \cite{AGS}, where we develop a new approach to large time estimate exclusive to the families of de Rham and signature operators. On the contrary, a family of signature operators twisted by a fibrewise flat bundle has to be treated as a general Dirac operator \cite{BH3}.

Next we investigate the $L^2$-rho form in relation to the space $\mathcal R^+(M/B)$ of positive scalar curvature vertical metrics for a fibre bundle of spin manifolds. For this purpose, the Dirac families $\cDi$ involved are uniformly invertible by Lichnerowicz formula, so that the definition of the $L^2$-rho form does not require Theorem \ref{th-eta}, but follows from classical estimates. 
Here the $L^2$-rho form is always closed, and we prove the first step in order to use this invariant for the study of $\mathcal R^+(M/B)$, namely that the class $[\hat{\rho}_{(2)}(\cDi)]$ is the same for metrics in the same concordance classes of $\mathcal R^+(M/B)$ (Prop.\ref{conc}). The action of a fibrewise diffeomorphism is also taken into account.

Along the lines of \cite{PS1} we can expect that if $\G$ is torsion-free and satisfies the Baum--Connes conjecture, then the $L^2$-rho class of a family of odd signature operators is an oriented $\G$- fibrewise homotopy invariant, and that $[\hat{\rho}_{(2)}(\tilde{\cDi}_{\hat g})]$ vanishes correspondingly to a vertical metric $\hat{g}$ of positive scalar curvature.
%

{\bf Acknowledgements}
This work was part of my researches for the doctoral thesis. I would like to thank Paolo Piazza for having suggested the subject, for many interesting discussions and for the help and encouragement. I wish to express my gratitude to Moulay-Tahar Benameur for many interesting discussions.




\section{Geometric families in the $L^2$-setting}

We recall local index theory's machine, here adapted to the following $L^2$-setting for families.

\begin{definition}\label{norcov}
Let $\tilde{\pi}\colon\tilde{M}\rightarrow B$ be a smooth fibre bundle, with typical fibre $\tilde{Z}$ connected, and let $\G$ be a discrete group acting fibrewise freely and properly discontinuosly on $M$, such that the quotient $M=\tilde{M}/\G$ is a fibration $\pi\colon M\rightarrow B$ with compact fibre $Z$. Let $p\colon \tilde{M}\rightarrow \tilde{M}/\G=M$ denote the covering map. This setting will be called a \emph{normal covering of the fibre bundle $\pi$} and will be denoted with the pair $(p\colon \tilde{M}\rightarrow M, \pi\colon M\rightarrow B)$.
\end{definition}


Let $\pi \colon M\rightarrow B$ be endowed with the structure of a \textit{geometric family} $(\pi\colon M\rightarrow B, g^{M/B}, \mathcal V, E)$, meaning by definition:\begin{itemize}
  \item $g^{M/B}$ is a given metric on the vertical tangent bundle $T(M/B)$
  \item $\mathcal V$ the choice of a smooth projection $\mathcal V\colon TM\rightarrow T(M/B)$ (equivalently, the choice of a horizontal complement $T^H M=\Ker \mathcal V$)
  \item $E\rightarrow M$ is a \emph{Dirac bundle}, i.e. an Hermitian vector bundle of vertical Clifford modules, with unitary action $c\colon\C l(T^*(M/B),g^{M/B})\rightarrow \End(E)$, and Clifford connection $\nabla^E$.
\end{itemize}

To a gemetric family it is associated a \emph{family $\mathcal{D}=(D_b)_{b\in B}$ of Dirac operators} along the fibres of $\pi$, $D_b
=c_b\circ \nabla^{E_b} \colon \mathcal{C}^\infty(M_b, E_b)\rightarrow \mathcal{C}^\infty(M_b, E_b)$, where $M_b=\pi^{-1}(b)$, and $E_b:=E_{|M_b}$.

If we have a normal $\G$-covering $p\colon \Mt\rightarrow M$ of the fibre bundle $\pi$, the pull back of the geometric family via $p$ gives a \emph{$\G$-invariant geometric family} which we denote $(\tilde{\pi} \colon  \tilde{M}\rightarrow B, p^*g^{M/B}, \tilde{\mathcal V}, \Et)$.

\subsubsection{The Bismut superconnection}
The structure of a geometric family gives a distinguished metric connection $\nabla^{M/B}$ on $T(M/B)$, defined as follows: fix any metric $g_B$ on the base and endow $TM$ with the metric $g=\pi^*g_B\oplus g_{M/B}$; 
let $\nabla^g$ the Levi-Civita connection on $M$ with respect to $g$; the connection $\nabla^{M/B}:=\mathcal{V}\nabla^g\mathcal{V}$ on the vertical tangent does not depend on $g_B$ (\cite[Prop. 10.2]{BGV}
). 

When $X\in \mathcal{C}^\infty(B,TB)$, let $X^H$ denote the unique section of $T^HM$ s.t. $\pi_*X^H=X$.
 For any $\xi_1, \xi_2\in \mathcal{C}^\infty(B, TB)$ let
$
T(\xi_1,\xi_2):=[\xi^H_1,\xi^H_2]-[\xi_1,\xi_2]^H
$
and let $\de\in \mathcal{C}^\infty(M, (T^HM)^*)$ measuring the change of the volume of the fibres
$\mathcal{L}_{\xi^H}\vol=:\de(\xi^H)\vol
$.
Following the notation of \cite{BGV}, in formulas in local expression we denote as $e_1,\dots , e_n$ a local orthonormal base of the vertical tangent bundle; $f_1,\dots f_m$ will be a base of $T_y B$ and $dy^1,\dots, dy^m$ will denote the dual base. The indices $i,j,k..$ will be used for vertical vectors, while $\a,\b,\dots$ will be for the horizontal ones.
The $2$-form 
$ c(T)=\sum_{\a<\b}(T(f_\a,f_\b), e_i)e_i dy^\a dy^\b
$ has values vertical vectors. Using the vertical metric, $c(T)(f_\a, f_\b)$ can be seen as a cotangent vertical vector, hence it acts on $E$ via Clifford multiplication.

Let $\mathcal{H}\rightarrow B$ be the infinite dimensional 
 bundle with fibres $\mathcal{H}_b=\mathcal{C}^\infty(M_b,E_b)$. Its space of sections is given by
$ \mathcal{C}^\infty (B,\mathcal{H})= \mathcal{C}^\infty (M,E)$. We denote $\O(B,\mathcal{H}):=\mathcal{C}^\infty(M,\pi^*(\Lambda T^* B)\otimes E)$. 
Let $\nabla^{\mathcal{H}}$ be the connection on  $\mathcal{H}\rightarrow B$ defined by
$
\nabla^{\mathcal{H}}_{U}\xi=\nabla^E_{U^H}\xi+\frac{1}{2}\delta(\xi^H)
$
where $\xi$ is on the right hand side is regarded as a section of $E$. $\nabla^{\mathcal{H}}$ is compatible with the inner product 
$<s,s'>_b:=\displaystyle\int_{Z_b}h^E(s,s')\vol{}_b\,$, with $s, s'\in \mathcal{C}^\infty (B,\mathcal H)$, and $h^E$ the fixed metric on $E$. 

\paragraph{\textit{Even dimensional fibre}}
When $\dim F=2l$ the bundle $E$ is naturally $\mathbb{Z}_2$-graded by chiraliry, $E=E^+\oplus E^-$, and $\mathcal D$ is odd. Correspondingly, the infinite dimensional bundle is also $\mathbb{Z}_2$-graded: $\mathcal{H}=\mathcal{H}^+\oplus \mathcal{H}^-$. The \textit{Bismut superconnection} adapted to $\mathcal{D}$ is the superconnection $\ds
\mathbb{B}=\nabla^{\mathcal{H}}+\mathcal{D}-\frac{c(T)}{4}$
on $\mathcal H$.

The corresponding bundle for the covering family $\tilde{\pi}$ is denoted $\tilde{\mathcal H}\rightarrow B$ where the same construction for the family $\Mt\rightarrow B$ gives the Bismut superconnection $\mathbb{\tilde{B}}=\ds\nabla^{\mathcal{\tilde{H}}}+ \mathcal{\tilde{D}}-\frac{c(\tilde{T})}{4}
$, adapted to $\mathcal{\tilde{D}}$. It is $\G$-invariant by construction, being 
the pull-back via $p$ of $\mathbb{B}$.

\paragraph{\textit{Odd dimensional fibre}}
When $\dim Z=2l-1$, the appropriate notion is the one of $\Cli$-superconnection, as introduced by
Quillen in \cite[sec. 5]{Q}.  
Let $\Cli$ the Clifford algebra $\Cli=\mathbb{C}\oplus \mathbb{C}\s$, where $\s^2=1$, and consider $\End E\otimes \Cli$, adding  
therefore the extra Clifford variable $\sigma$. 
On $\End(E_b)\otimes \Cli=\End_\s(E_b\oplus E_b)$ define the supertrace
$\tr{}^\s(A+B\s):=\tr{}B$,
%
extended then to $\tr{}^\s\colon  \mathcal{C}^\infty(M, \pi^* \Lambda^*B\otimes \End E)\rightarrow \O(B)$ as usual by 
$
\tr{}^\s(\o\otimes (a+b\s))=\o \tr b
$, 
 for $\o\in \mathcal{C}^\infty(B, \Lambda T^*B)$, $\forall a,b\in \mathcal{C}^\infty(B,\End E)$.

The family $\mathcal{D}$, as well as $c(T)$ are even degree elements of the algebra $\mathcal{C}^\infty(B, \End \mathcal{H}\otimes \Cli\hat{\otimes}\Lambda T^*B)$. On the other hand, $\nabla^\mathcal H$ is odd.
By definition, the \emph{Bismut $\Cli$-superconnection} adapted to the family $\mathcal{D}$ is the operator  of odd total degree
$\mathbb{B}^\s:=\ds\mathcal{D}\s+\tilde{\nabla}^u-\frac{c(T)}{4}\s$.
\paragraph {\textit{Notation.}} In the odd case we will distinguish between the $\Cli$-superconnection defined above $\mathbb{B}^\s$ acting on $\Omega(B,\mathcal{H})\,\hat{\otimes}\,\Cli$, and the differential operator 
$\mathbb{B}\colon \Omega(B,\mathcal{H})\rightarrow \Omega(B,\mathcal{H})$ given by $
\mathbb{B}:=\mathcal{D}+\nabla^\mathcal H-\frac{c(T)}{4}
$, which is not a superconnection but is needed in the computations.


\subsection{The heat operator for the covering family} \label{sec.inv}

In this section we briefly discuss the construction of the heat operator $e^{-\tilde{\mathbb{B}}^2}$, which can be easily performed combining the usual construction for compact fibres families in \cite[Appendix of Chapter 9]{BGV}, with Donnelly's construction for the case of a covering of a compact manifolds \cite{Do}.
We integrate notations of \cite[Ch. 9-10]{BGV} with the ones of our appendix \ref{app1}. We refer to the latter for the definitions of the spaces of operators used the rest of this section.

Let $\mathcal{C}^\infty(B,\Diff_\G(\tilde{E}))$ the algebra of smooth maps $D\colon B\rightarrow \Diff_\G(\tilde{E})$ satisfying that $\forall z\in B$, $D_z$ is a $\G$-invariant differential operator on $\Mt_z$, with coefficients depending smoothly on the variables of $B$.
In the same way, let $\mathcal{N}=\mathcal{C}^\infty(B,\Lambda T^*B\otimes\Op_\G^{-\infty}(\tilde{E}))=\O(B,\Op_\G^{-\infty}(\tilde{E}))$ the space of smooth maps $A \colon B \rightarrow\Lambda T^*B\otimes \Op_\G^{-\infty}(\tilde{E})$.
$\mathcal N$ contains families of $\G$-invariant operators of order $-\infty$ with coefficients differential forms, hence $\mathcal{N}$ is filtered by $\mathcal{N}_i= \mathcal{C}^\infty(B,\bigoplus_{j\geq i}\Lambda^j T^*B\otimes \Op_\G^{-\infty}(\tilde{E}))$.
The curvature of $\mathbb{\tilde{B}}$ is a family $\mathbb{\tilde{B}}^2\in \O(B, \Diff^2_\G(\tilde{E}))$ and can be written as $\mathbb{\tilde{B}}^2=
\tilde{D}^2-\tilde{C}$, with 
$\tilde{C}\in \O^{\geq 1}(B, \Diff_\G^1(\Et))$.

\subsubsection{Definition and construction}
For each point $z\in B$ the operator $e^{-t\Bt^2_{z}}$ is by definition an 
the one whose Schwartz kernel $\tilde{p}^z_t(x,y)\in \tilde{E}_x\otimes \tilde{E}^*_y\otimes \Lambda T_z^*B$ is the fundamental solution of the heat equation, i.e.
\begin{itemize}
  \item $\tilde{p}_t^z(x,y)$ is $C^1$ in $t$, $C^2$ in $x, y$; 
  \item $\ds \frac{\partial }{\partial t} \tilde{p}_t^z(x,y)+\Bt^2_{z,II} \tilde{p}_t^z(x,y)=0$ where $\Bt_{z,II}$ means it acts on the second variable;
\item $\ds \lim_{t\rightarrow 0} \tilde{p}_t^z(x,y)=\delta(x,y)$
\item $\forall T>0$ $\forall t\leq T$ $\exists \,c(T): \n{\partial_t^i\partial_x^j\partial_y^k p_t(x,y)}\leq ct^{-\frac{n}{2}-i-j-k} e^{-\frac{d^2(x,y)}{2}}, \,\,0\leq i,j,k\leq 1$.
\end{itemize}
Its construction is as follows: pose
\begin{equation}\label{heat}
e^{-t\Bt_{z}^2}:=e^{-t\tilde{D}_z^2}+\sum_{k>0} \int_{\triangle_k} t^k\underbrace{e^{-\s_0 t \Dt^2_z } \tilde{C} e^{-\s_1 t \Dt^2_z} \dots \tilde{C} e^{-\s_k t \tilde{D}^2_z}}_{I_k} d\s_1\dots d\s_k
\end{equation}
Since $\forall \s=(\s_0,\dots,\s_k)$ there exists $\s_i>\frac{1}{k+1}$, then each term $I_k\in \Lambda T^*_z B\otimes \Op^{-\infty}(\Et_z)$ and so does $e^{-t\Bt_{z}^2}$.
Let $\tilde{p}^z_t(x,y)=[e^{-\Bt^2_{t,z}}](x,y)$ be the Schwartz kernel of the operator \eqref{heat}. Using arguments of   \cite[theorems 9.50 and 9.51]{BGV}, one proves that $\tilde{p}_t^z(x,y)$ is smooth in $z\in B$ so that one can conclude
$e^{-\Bt}\in\O(B, \Op^{-\infty}_\G)$.

The next property, proved in \cite{Do} and \cite{Do2}, is needed in the $t\rightarrow 0$ asymptotic.
For $t<T_0$
\begin{equation}\label{decad.esp}
\left|[e^{-t\Bt^2}](\tilde{x},\tilde{y})\right|\leq c_1t^{-\frac{n}{2}}\ds e^{-c_2\frac{d^2(\tilde{x},\tilde{y})}{t}}
\end{equation}
\subsection{Transgression formul\ae, eta integrands}
For $ t>0$ let $\de_t\colon \Omega(B,\mathcal{H})\rightarrow \Omega(B,\mathcal{H})$ the operator which on $\Omega^i(B,\mathcal{H})$ is multiplication by $t^{-\frac{i}{2}}$. Then consider the rescaled superconnection $\ds\mathbb{B}_t=t^{\frac{1}{2}}\de_t\mathbb{B}\de_t^{-1}=\nabla^{\mathcal{H}}+\sqrt t \mathcal{D}-c(T)\frac{1}{4\sqrt t}$.
\subsubsection{Even dimensional fibre}
From \eqref{trace} we have
$\ds\frac{d}{dt}\Str{}_{\G} e^{-\Bt^2_t}=-d \Str{}_\G\left(\frac{d \Bt}{dt}e^{-\Bt^2_t}\right)$ which on a finite interval $(t,T)$ gives the \textit{transgression formula}
\begin{equation}\label{transg}
\Str{}_\G\left( e^{-\tilde{\mathbb{B}}^2_T}\right)-\Str{}_\G\left( e^{-\tilde{\mathbb{B}}^2_t}\right)=-d \int_t^T\Str{}_\G\left( \frac{d \tilde{\mathbb{B}}_s}{ds}e^{-\tilde{\mathbb{B}}_s^2}\right)ds\end{equation}

\subsubsection{Odd dimensional fibre}
Here it is convenient to use that
$\ds\tr{}^\s_\G e^{-(\mathbb{\tilde{B}}^\s_t)^2}=\tr{}^{odd}_\G e^{-\mathbb{\tilde{B}}_t^2}$,  (from \cite{Q} and \eqref{trace}), where $\tr^{odd}$ means we take the odd degree part of the resulting form. Then taking the odd part of the formula $\ds
\frac{\partial}{\partial t}\tr{}_\G e^{-\mathbb{B}_t^2}=-d\tr{}_\G \left(\frac{\partial \mathbb{B}_t}{\partial t} e^{-\mathbb{B}_t^2}\right)$ 
\begin{equation}\label{transgs}
\Tr{}^{odd}_\G\left( e^{-\tilde{\mathbb{B}^\s}^2_T}\right)-\Tr{}^{odd}_\G\left( e^{-\Bt^2_t}\right)=-d \int_t^T\Tr{}^{even}_\G\left( \frac{d \Bt_s}{ds}e^{-\Bt_s^2}\right)ds \end{equation}

\begin{remnot}\label{conv-forms} Since we wish now to look at the limits as $t\rightarrow 0$ and $t\rightarrow \infty$ in \eqref{transg} and \ref{transgs}, let us make precise what the convergences on the spaces of forms are, and for families of operators. On $\O(B)$ we consider the topology of convergence on compact sets. We say a family of forms $\o_t\stackrel{C^0}{\rightarrow} \o_{t_0}$ as $t\rightarrow t_0$ if $\forall K\stackrel{cpt}{\subseteq} B$ $\,\sup_{z\in K}\n{\o_t(z)-\o_{t_0}(z)}_{\Lambda T_z^*B}\rightarrow 0$. We say $\o_t\stackrel{C^1}{\rightarrow} \o_{t_0}$ if the convergence also hold for first derivatives of $\o_t$ with respect to the base variables. 
We say $\o_t=\mathcal O(t^\de)$ as $t\rightarrow \infty$ if $\exists$ a constant $C=C(K)$ : $\,\sup_{z\in K}\n{\o_t(z)-\o_{t_0}(z)}_{\Lambda T_z^*B}\leq C t^\de$. We say $\o_t\stackrel{C^1}{=}\mathcal O(t^\de)$ if also the first derivatives with respect to base directions are $\mathcal O(t^\de)$.

For a family $T_t\in U\mathcal{C}^\infty (B, \Op^{-\infty}(\Et))$ we say $T_t\stackrel{\mathcal{C}^k}{\rightarrow} T_{t_0}$ as $t\rightarrow t_0$ if $\forall K\stackrel{cpt}{\subseteq} B$, $\forall r,s \in \mathbb{Z}$ $\sup_{z\in K}\n{T_t(z)-T_{t_0}(z)}_{r,s}\rightarrow 0$ together with derivatives up to order $k$ with respect to the base variables. 

On the space of kernels $U\mathcal{C}^\infty(\Mt\times_B\Mt, \Et\XBox\Et^*\otimes \pi^*\Lambda T^*B)$, we say $k_t\rightarrow k_{t_0}$ if $\forall \phi \in C_c^\infty(B)$  $\n{(\pi^*\phi(x))(k_t(x,y)-k_{t_0}(x,y))}_k\rightarrow 0$. 

We stress that from \eqref{C0sob} the map  $\,\O(B,\Op_\G^{-\infty}(\Et))\rightarrow U\mathcal{C}^\infty(\Mt\times_B\Mt, \Et\XBox\Et^*\otimes \pi^*\Lambda T^*B)$, $T\mapsto [T]$ is continuos.
\end{remnot}
\subsection{The $t\rightarrow 0$ asymptotic}\label{sec.loc}
 \begin{proposition}\label{teoloc}
\begin{eqnarray*}
\lim_{t\rightarrow 0}  \Str{}_\G \left(e^{-\mathbb{\mathbb{\tilde{B}}}_t^2}\right) =\int_{M/B} \hat{A}(M/B)\ch E/S \;\; \text{if}\; \dim \tilde{Z}=\text{ even}
\\
\lim_{t\rightarrow 0}  \tr{}^{odd}_\G \left(e^{-\mathbb{\mathbb{\tilde{B}}}_t^2}\right) =\int_{M/B} \hat{A}(M/B)\ch E/S 
\;\; \text{if}\; \dim \tilde{Z}=\text{ odd}
\end{eqnarray*}
\end{proposition}
The result is proved exactly as in the classic case of compact fibres, together with the following argument of \cite[Lemma 4, pag. 4]{Lo2}:
\begin{lemma}\cite{Lo2}
$\exists A>0, c>0$ s.t. 
\begin{equation*}
\left|[e^{-\mathbb{B}^2_t}](\pi(\tilde x), \pi(\tilde x))-[e^{-\Bt^2_t}](\tilde x, \tilde x)\right|=\mathcal{O}(t^{-c}e^{-\frac{A}{t}})
\end{equation*}
\end{lemma}
For the proof of the lemma see \cite{Lo}, or also \cite{AzT}, \cite{GR}.
With the same technique we deduce
\begin{proposition}\label{intloc}
The differential forms $\Str_\G \ds \left(\frac{d\mathbb{\tilde{B}}_t}{dt}e^{-\mathbb{\tilde{B}}_t^2}\right)$ and $\tr{}^\s_\G \ds\left( \frac{d\mathbb{\tilde{B}^\s}_t}{dt}e^{-(\mathbb{\tilde{B}}^\s_t)^2}\right)$ are integrable on $[0,1]$, uniformly on compact subsets.
\end{proposition}

\begin{proof} The proof is as in \cite[Ch.10, pag. 340]{BGV}. We reason for example in the even case. Consider the rescaled superconnection $\mathbb{\tilde{B}}_s$ as a one-parameter family of superconnections, $s\in \mathbb{R}^+$, and construct the new family $\breve{M}=\Mt\times \mathbb{R}^+\rightarrow B\times \mathbb{R}^+=:\breve{B}$. On $\breve{E}=\tilde{E}\times \mathbb{R}^+$ there is a naturally induced family of Dirac operators whose Bismut superconnection is $\breve{\mathbb{B}}=\mathbb{\tilde{B}}_s+d_{\mathbb{R}^+}-\frac{n}{4s}ds$, and its rescaling is 
$\ds
\breve{\mathbb{B}}_t=\tilde{\mathbb{B}}_{st}+d_\mathbb{R^+}-\frac{n}{4s}ds
$.
Its curvature is  
$\ds
\breve{\mathbb{B}}^2_t=\tilde{\mathbb{B}}^2_{st}+t\frac{d\tilde{\mathbb{B}}_s}{ds} \wedge ds
$, so that 
$$
e^{-\breve{\mathbb{B}}_t^2}=e^{-\tilde{\mathbb{B}}^2_{st}}-\int_0^1 e^{-u\tilde{\mathbb{B}}^2_{st}}t\frac{d\tilde{\mathbb{B}}_s}{ds} e^{-(1-u)\mathbb{B}^2_{st}}\wedge ds=e^{-\tilde{\mathcal{F}}_{st}}-\frac{\partial\tilde{\mathbb{B}}_{st}}{\partial s} e^{-\tilde{\mathbb{B}}_{st}}\wedge ds.
$$Then
\begin{equation}
\label{strtr}
\Str{}_\G\left(e^{-\breve{\mathbb{B}}^2_t}\right)=\Str{}_\G(e^{-\tilde{\mathbb{B}^2_{st}}})-\Str{}_\G\left(\frac{\partial\tilde{\mathbb{B}}_{st}}{\partial s} e^{-\tilde{\mathbb{B}}_{st}}\right) ds
\end{equation}
At $t=0$ we have the asymptotic expansion 
$
\Str{}_\G(e^{-\breve{\mathbb{B}}_t})\sim \sum_{j=0}^\infty t^{\frac{j}{2}}(\Phi_{\frac{j}{2}}-\alpha_{\frac{j}{2}}ds)
$, without singular terms. 
Computing \eqref{strtr} in $s=1$, since $\displaystyle\frac{\partial\tilde{\mathbb{B}}_{st}}{\partial s}=t \frac{\partial\tilde{\mathbb{B}}_{s}}{\partial s}\,$, one has
$\;\displaystyle 
\Str{}_{\G}\left(t\frac{\partial\tilde{\mathbb{B}}_{s}}{\partial s} e^{-\mathcal{\tilde{F}}_t}\right)\sim \sum_{j=0}^\infty t^{\frac{j}{2}}\a_{\frac{j}{2}}
$, 
and therefore $
\Str{}_{\G}\left(\frac{\partial \tilde{\mathbb{B}}_{s}}{\partial s} e^{-\mathcal{\tilde{F}}_t}\right)\sim \sum_{j=0}^\infty t^{\frac{j}{2}-1}\a_{\frac{j}{2}}$.
Let's compute $\a_0$. From the local formula
\begin{equation}
\label{loct}\Phi_0-\a_0 ds=\lim_{t\rightarrow 0}\Str{}_{\G}\left(e^{-\breve{\mathcal{F}}_t}\right)=
\int_{\breve{M}/\breve{B}} \hat{A}(\breve{M}/\breve{B})
\end{equation}
since $\breve{M}_{(z,s)}=\Mt_z\times\{s\}$ and the differential forms are pulled back from those on $\tilde{M}\rightarrow B$, then the right hand side of \eqref{loct} does not contain $ds$ so that $\a_0=0$.
This implies that $
\Str{}_\G(\ds \frac{d\mathbb{\tilde{B}}_t}{dt}e^{-\mathbb{\tilde{B}}_t^2})\sim \sum_{j=1}^\infty t^{\frac{j}{2}-1}\a_{\frac{j}{2}}$.
\end{proof}

\section{The $L^2$-eta form}
\label{sec.wea}
We prove in Theorem \ref{th-eta} the well definiteness of the $L^2$-eta form $\hat{\eta}_{(2)}(\mathcal{\tilde{D}})$ under opportune regularity assumptions.
We make use of the techniques of \cite{HL}. 

\subsection{The family Novikov--Shubin invariants}\label{N--S}
The $t\rightarrow \infty$ asymptotic of the heat kernel is controlled by the behaviour of the spectrum near zero.
Let $\tilde{P}=(\tilde{P}^z)_{z\in B}$ the family of projections onto $\ker \mathcal{\tilde{D}}$ and let $\tilde{P}_\ep=\chi_{(0,\ep)}(\tilde{\mathcal{D}})$ be the family of spectral projections relative to the interval $(0,\ep)$; denote  $\tilde{Q}_\ep=1-\tilde{P}_\ep-\tilde{P}$.

For any $z\in B$ the operator $\tilde{D}_z$ is a $\Gamma$-invariant unbounded operator: let $\tilde{D}^2_z=\int \la dE^z(\la)$ be the spectral decomposition of $\tilde{D}_z^2$, and $N^z(\la)=\tr_\Gamma E^z(\la)$ its spectral density function \cite{GS}.
Denote $b^z=\tr_\G \tilde{P}^z$. Then $N^z(\ep)=b^z+\trG \tilde{P}^z_\ep$ and from \cite{ES} the behaviour of 
$\theta^z(t)=\trG(\exp(-t\tilde{D}_z))$ at $\infty$ is governed by
\begin{equation}
\a_z=\sup\{a: \theta^z(t)=b^z+\mathcal O(t^{-a})\}=\sup \{a: N^z(\ep)=b^z+\mathcal O(\ep^a)\}\end{equation}
where $\a_z$ is called the \emph{Novikov--Shubin invariant} of $\tilde{D}_z$.

We shall later impose conditions on $\a_z$ uniformly on compact subset of $B$, so we introduce the following definition from \cite{GR}: let $K\subset B$ be a compact, define $\a_K:=\inf_{z\in K}\a_z$. We call $\{\a_K\}_{K\subset B}$ the \emph{family Novikov--Shubin invariants} of the fibre bundle $\tilde{M}\rightarrow B$.

By results of Gromov and Shubin \cite{GS}, when $\tilde{D}^2_z$ is the Laplacian, $\a_z$ is a $\G$-homotopy invariant of $\tilde{M}_z$ \cite{GS}, in particular it does not depend on $z$. In that case $\a_z$ is locally constant on $B$. For a general Dirac type operator this is not true and we need to use the $\a_K$'s.

\begin{definition}\cite{HL} We say the family $\mathcal{\tilde{D}}$ has \emph{regular spectral projections} if $\tilde{P}$ and $\tilde{P}_\ep$ are smooth with respect to $z\in B$, for $\ep$ small, and $\nabla^{\tilde{\mathcal{H}}}\tilde{P}, \nabla^{\tilde{\mathcal{H}}}\tilde{P}_\ep$ are in $ \mathcal N$ and are bounded independently of $\ep$.
We say that the family \emph{$\mathcal{\tilde{D}}$ has regularity $A$}, if $\forall K\stackrel{cpt}{\subseteq}  B$ it holds $\a_K\geq A$.
\end{definition}
\begin{remark}To have regular projections is a strong condition, difficult to be verified in general. The family of signature operators verifies the smoothness of $\tilde{P}$ \cite[Theorem 2.2]{GR} but the smoothness of $\tilde{P}_\ep$ is not clear even in that case.
\end{remark}
The large time limit of the superconnection-Chern character $\Str_\G e^{-\tilde{\mathbb{B}}_t^2}$ is computed in \cite[Theorem 5]{HL}. Specializing to our $L^2$-setting it says the following. 
\begin{theorem}\label{thHL}
\cite{HL} Let  $\tilde{\nabla}_0=\tilde{P}\nabla^{\tilde{\mathcal H}}\tilde{P}$. If $\mathcal{\tilde{D}}$ has regular projections and regularity $>3\dim B$,
$$
\lim_{t\rightarrow \infty} \Str{}_\G(e^{-\mathbb{\tilde{B}}_t^2})=\Str{}_\G e^{-\tilde{\nabla_0}^2}.
$$  
\end{theorem}
\subsection{The $L^2$-eta form} 
We now use the same techniques of \cite{HL} to analyse the transgression term in \eqref{transg} and define the secondary invariant $L^2$ eta form.
We prove
\begin{theorem}\label{th-eta}
If $\mathcal{\tilde{D}}$ has regular spectral projections and regularity $>3(\dim B+1)$, then $ \Str_\G \ds \left(\frac{d\mathbb{\tilde{B}}_t}{dt}e^{-\mathbb{\tilde{B}}_t^2}\right)=\mathcal O(t^{-\de-1})$, for $\de>0$. The same holds for $\tr^{even}_\G \ds \left(\frac{d\mathbb{\tilde{B}}_t}{dt}e^{-\mathbb{\tilde{B}}_t^2}\right) $.
\end{theorem}
We start with some remarks and lemmas. In particular we shall repeatedly use the following.
\begin{remark}
Let $T\in \mathcal{N} $. From lemma \ref{appker}, $\forall z\in B\;$ its Schwartz kernel $[T_z]$ satisfies that for sufficiently large $l$, $\exists\, c_l^z$ such that
$
\forall x,y\in \Mt_z \;\;\; \left|\;[T_z](x,y)\;\right|\leq c_l^z \n{T_z}_{-l,l}
$
Therefore an estimate of $ \n{T_z}_{-l,l}
$ produces directly via an estimate of $\Tr_\G T_z$.
\end{remark}
\paragraph{\textit{Notation.}}Since in this section we are dealing only with the family of operators on the covering, to simplify the notations let's call $\mathcal{\tilde{D}}=D$, removing all tildes.
Pose
$$
\mathbb{B}_\ep:=(P+Q_\ep)\mathbb{B}(P+Q_\ep)+P_\ep\mathbb{B}P_\ep
$$
$$
A_{\ep}=\mathbb{B}-\mathbb{B}_\ep
$$
and write the rescaled operators as\begin{equation}
\label{aet}
\mathbb{B}_{\ep,t}=(P+Q_\ep)(\mathbb{B}_t-\sqrt{t}D)(P+Q_\ep)+\sqrt{t}D+P_\ep(\mathbb{B}_t-\sqrt{t}D)P_\ep
\end{equation}
$$
A_{\ep,t}=(P+Q_\ep)(\mathbb{B}_t-\sqrt{t}D)P_\ep+P_\ep(\mathbb{B}_t-\sqrt{t}D)(P+Q_\ep)
$$
Denote also $T_\ep=Q_\ep \mathbb{B}Q_\ep$ and $T_{\ep,t}=Q_\ep \mathbb{B}_t Q_\ep$ as in \cite{HL}.

We will need the following two lemmas from \cite{HL}. The first is the \lq\lq diagonalization" of $\mathbb{B}_\ep^2$ with respect to the spectral splitting of $\mathcal H$. 
\begin{lemma}\cite[Prop.6]{HL} 
Let $\mathcal{M}$ be the space of all maps $f\colon B\rightarrow \Lambda TB\otimes \End \tilde{\mathcal{H}}$.
There exists a measurable section $g_\ep\in \mathcal{M}$, with $g_\ep\in 1+\mathcal{N}_1$ such that
$$g_\ep \mathbb{B}^2_\ep g_\ep^{-1}=\left|\begin{array}{ccc}\nabla_0^2 & 0 & 0 \\0 & T_\ep^2 & 0 \\0 & 0 & (P_\ep\mathbb{B}P_\ep)^2
\end{array}\right|\;\;\;\;\text{mod }\;\;\;\;\left|\begin{array}{ccc}\mathcal{N}_3 & 0 & 0 \\0 & \mathcal{N}_2 & 0 \\0 & 0 & 0\end{array}\right|.$$
\end{lemma}
The diagonalization procedure acts on $(P\oplus Q_\ep)\mathcal H$, in fact $g_\ep$ has the form $g_\ep=\hat{g}_\ep\oplus 1$, with $\hat{g}_\ep$ acting on $(P\oplus Q_\ep)\mathcal H$.
From this lemma we get $\mathbb{B}_{\ep,t}^2=t\de_t\mathbb{B}_\ep^2\de_t^{-1}=$
\begin{multline*}=t\de_t g_\ep^{-1}\left(\left|\begin{array}{ccc}\nabla_0^2 & 0 & 0 \\0 & T^2_\ep & 0 \\0 & 0 & (P_\ep\mathbb{B}_tP_\ep)^2\end{array}\right| + \left|\begin{array}{ccc}\mathcal{N}_3 & 0 & 0 \\0 & \mathcal{N}_2 & 0 \\0 & 0 & 0\end{array}\right|\right)g_\ep \de_t=\\
=\de_t g_\ep^{-1}\de_t^{-1}\left|\begin{array}{ccc}t\de_t(\nabla_0^2+\mathcal{N}_3)\de_t^{-1} & 0 & 0 \\0 & t\de_t(T^2_\ep+\mathcal{N}_2)\de_t^{-1} & 0 \\0 & 0 & P_\ep\mathbb{B}_tP_\ep
\end{array}\right|\de_t g_\ep\de_t^{-1}.
\end{multline*}
The next lemma gives an estimate of the terms which are modded out.
\begin{lemma}\label{est-rest}\cite[lemma 9]{HL}
If $A\in \mathcal{N}_k$ is a residual term in the diagonalization lemma or is a term in $g_\ep-1$ or $g_\ep^{-1}-1$, then, posing $\ep=t^{-\frac{1}{a}}$, $A_t:=\de_t A\de_t^{-1}$ verifies: $\forall r,s$
$$
\n{A_t}_{r,s}=\mathcal{O}(t^{-\frac{k}{2}+\frac{k}{a}})\;\;\;\text{as }\; t\rightarrow \infty.
$$
\end{lemma}
The lemma implies that at place (1,1) in the diagonalized matrix above we get $\nabla_0^2+\mathcal{O}(t^{-\frac{3}{2}+\frac{3}{a}+1})=\mathcal{O}(t^{-\frac{1}{2}+\frac{3}{a}})$. 
To have $-\frac{1}{2}+\frac{3}{a}<0$ we \textbf{take} $\mathbf{a>6}$. The term at place (2,2) gives $T_{\ep,t}^2+\mathcal{O}(t^\frac{2}{a})$. Then 
$$
\mathbb{B}_{\ep,t}^2=\de_t g_\ep^{-1}\de_t^{-1}\left|\begin{array}{ccc}\nabla_0^2+\mathcal{O}(t^{-\g}) & 0 & 0 \\0 & T^2_\ep+\mathcal{O}(t^{\frac{2}{a}}) & 0 \\0 & 0 & (P_\ep\mathbb{B}P_\ep)^2
\end{array}\right|\de_t g_\ep\de_t^{-1}\;\;,\;\text{with}\;\;\;\;\; \g>0
$$
Now since $g_\ep=\hat{g}_\ep\oplus 1$
$$
\mathbb{B}_{\ep,t}^2=\left|\begin{array}{c|c}  \de_t \hat{g}_{\ep}^{-1}\de_t^{-1} \left|\begin{array}{cc}\nabla_0^2+\mathcal{O}(t^{-\g}) & 0 \\0 & T^2_{\ep,t}+\mathcal{O}(t^{\frac{2}{a}})\end{array}\right|\de_t \hat{g}_{\ep}\de_t^{-1}     & 0 \\\hline 0 & P_\ep\mathbb{B}P_\ep
\end{array}\right|
$$
Observe that since $g_\ep-1, g_\ep^{-1}-1\in\mathcal{N}_1$, we have
$
\de_t \hat{g}_{\ep}^{-1}\de_t^{-1}=\Id+\left|\begin{array}{cc}1&1\\1&1\end{array}\right|\mathcal{O}(t^{-\frac{1}{2}+\frac{1}{a}})
$.
%
%
Denote $w:=\mathcal{O}(t^{-\frac{1}{2}+\frac{1}{a}})$. Then
\begin{multline*}
\de_t \hat{g}_{\ep}^{-1}\de_t^{-1} \left|\begin{array}{cc}\nabla_0^2+\mathcal{O}(t^{-\g}) & 0 \\0 & T^2_{\ep,t}+\mathcal{O}(t^{\frac{2}{a}})\end{array}\right|\de_t \hat{g}_{\ep}\de_t^{-1}=  \\
=
 \left|\begin{array}{cc}1+w & w \\w & 1+w\end{array}\right|\left|\begin{array}{cc}\nabla_0^2+\mathcal{O}(t^{-\g}) & 0 \\0 & T^2_{\ep,t}+\mathcal{O}(t^{\frac{2}{a}})\end{array}\right|\left|\begin{array}{cc}1+w & w \\w & 1+w\end{array}\right|.
\end{multline*}
Since $e^{-\nabla_0^2+\mathcal{O}(t^{-\g})}= e^{-\nabla_0^2}+\mathcal{O}(t^{-\g})$, then leaving $(P+Q_\ep)$ out of the notation
\begin{multline*}
e^{-\mathbb{B}_{\ep,t}^2}=\left|\begin{array}{cc}1+w & w \\ w & 1+w\end{array}\right| \left|\begin{array}{cc}e^{-\nabla_0^2}+\mathcal{O}(t^{-\g}) & 0 \\0 & e^{-T}\end{array}\right| \left|\begin{array}{cc}1+\th & w \\w & 1+w\end{array}\right|+ e^{-(P_\ep\mathbb B P_\ep)^2}=\\
= e^{-(P_\ep\mathbb B P_\ep)^2}+\textbf{A}+\textbf{B}
\end{multline*}
where
$$
\textbf{A}=\left|\begin{array}{cc}(1+w)^2e^{-\nabla_0^2} & w(1+w)e^{-\nabla_0^2} \\w(1+w)e^{-\nabla_0^2} & w^2e^{-\nabla_0^2}\end{array}\right|=\left|\begin{array}{cc}e^{-\nabla_0^2} & 0 \\0 & 0\end{array}\right|+ \left|\begin{array}{cc}\mathcal{O}(t^{-1+\frac{2}{a}}) & \mathcal{O}(t^{-\frac{1}{2}+\frac{1}{a}}) \\ \mathcal{O}(t^{-\frac{1}{2}+\frac{1}{a}}) & \mathcal{O}(t^{-1+\frac{2}{a}})\end{array}\right|
$$
$$
\textbf{B}=\left|\begin{array}{cc}(1+w)^2  \mathcal{O}(t^{-\g}) & w(1+w)[\mathcal{O}(t^{-\g})+e^{-T}] \\w(1+w)[\mathcal{O}(t^{-\g})+e^{-T}] & w^2\mathcal{O}(t^{-\g})+(1+w)^2 e^{-T}\end{array}\right|.
$$
\begin{proof}[Proof of theorem \ref{th-eta}]
To fix notation, say $Z$ is even dimensional. In the odd case use $\tr^{even}_{\G}$ instead of $\StrG$.

Let $K\subseteq B$ be a compact, and denote as $\b=\a_K$ the Novikov--Shubin invariant on it. 

Write $\mathbb{B}_t=\mathbb{B}_{\ep,t}+A_{\ep,t}$ as in \eqref{aet}, and define $\mathbb{B}_t(z)=\mathbb{B}_{t,\ep}+zA_{t,\ep}$, $z\in[0,1]$, so that by Duhamel's principle (for example \cite[eq. (3.10)]{HL})
$$
e^{-\mathbb{B}_t^2}-e^{-\mathbb{B}_{t,\ep}^2}=\int_{0}^1\frac{d}{dz}e^{-\mathbb{B}_t(z)^2}dz=-\int_0^1 \int_0^1e^{-(s-1)\mathbb{B}^2_t(z)}\frac{d\mathbb{B}_t^2(z)}{dz}e^{-s\mathbb{B}^2_t(z)}dsdz=: F_{\ep,t}
$$
Write then 
\begin{equation}
\label{2term}
\Str{}_\G(\frac{d\mathbb{B}_t}{dt}e^{-\mathbb{B}_t^2})=\underbrace{\Str{}_\G(\frac{d\mathbb{B}_{t,\ep}}{dt}e^{-\mathbb{B}_{t,\ep}^2})}_I+\underbrace{\Str{}_\G(\frac{d\mathbb{B}_{t}}{dt}F_{\ep,t})}_{II}
\end{equation}
For the family $\ds\frac{d\mathbb{B}_t}{dt}$ we shall use that
$\ds\frac{d \mathbb{B}_t}{dt}=\frac{1}{2\sqrt t}\left( \mathcal{D}+\frac{c(T)}{4t} \right)=\frac{1}{2\sqrt t}\mathcal{D}+\mathcal O(t^{-\frac{3}{2}})$, as in Remark \ref{conv-forms}.
\subsubsection{The term I}
\begin{multline*}\frac{d\mathbb{B}_{t}}{dt}e^{-\mathbb{B}_{t,\ep}^2}=\left(\left|\begin{array}{ccc}0 & 0 & 0 \\0 & t^{-\frac{1}{2}}Q_\ep DQ_\ep & 0 \\0 & 0 & t^{-\frac{1}{2}}P_\ep DP_\ep\end{array}\right|+\mathcal{O}(t^{-\frac{3}{2}})\right) \left(e^{-(P_\ep \mathbb B P_\ep)^2}+\textbf{A}+\textbf{B}\right) =
\\
=\left|\begin{array}{ccc}0 & 0 & 0 \\0 & t^{-\frac{1}{2}}Q_\ep DQ_\ep & 0 \\0 & 0 & t^{-\frac{1}{2}}P_\ep DP_\ep\end{array}\right|\left(\left|\begin{array}{ccc}e^{-\nabla_0^2} & 0 & 0 \\0 & 0 & 0 \\0 & 0 & 0\end{array}\right|+ \left|\begin{array}{ccc}\mathcal{O}(t^{-1+\frac{2}{a}}) & \mathcal{O}(t^{-\frac{1}{2}+\frac{1}{a}}) & 0\\ \mathcal{O}(t^{-\frac{1}{2}+\frac{1}{a}}) & \mathcal{O}(t^{-1+\frac{2}{a}})& 0\\ 0& 0&0\end{array}\right|\right)+
\\
+\left|\begin{array}{ccc}0 & 0 & 0 \\0 & t^{-\frac{1}{2}}Q_\ep DQ_\ep & 0 \\0 & 0 & t^{-\frac{1}{2}}P_\ep DP_\ep\end{array}\right|\left|\begin{array}{ccc}(1+w)^2 \mathcal O(t^{-\g}) & w(1+w)^2(\mathcal O(t^{-\g})+e^{-T}) & 0 \\w(1+w)^2(\mathcal O(t^{-\g})+e^{-T}) & w^2\mathcal O (t^{-\g})+(1+w)^2e^{-T} & 0 \\0 & 0 & 0\end{array}\right|+
\end{multline*}
\begin{multline*}
+\left|\begin{array}{ccc}0 & 0 & 0 \\0 & t^{-\frac{1}{2}}Q_\ep DQ_\ep & 0 \\0 & 0 & t^{-\frac{1}{2}}P_\ep DP_\ep\end{array}\right| e^{-(P_\ep \mathbb B P_\ep)^2}=
\\
=t^{-\frac{1}{2}}P_\ep D P_\ep e^{-(P_\ep \mathbb B P_\ep)^2}+\left|\begin{array}{ccc}0 & 0 & 0 \\t^{-\frac{1}{2}}Q_\ep D Q_\ep \mathcal O(t^{-\frac{1}{2}+\frac{1}{a}}) & Q_\ep D Q_\ep \mathcal O(t^{-\frac{3}{2}+\frac{2}{a}}) & 0 \\0 & 0 & 0\end{array}\right|+
\\
+\left|\begin{array}{ccc}0 & 0 & 0 \\t^{-\frac{1}{2}}Q_\ep D Q_\ep w(1+w) (\mathcal O(t^{-\g})+e^{-T}) & t^{-\frac{1}{2}}Q_\ep D Q_\ep (w^2 \mathcal O(t^{-\g})+(1+w)^2e^{-T}) & 0 \\0 & 0 & 0\end{array}\right|.
 \end{multline*}

The choice of $a>6$ implies $\frac{2}{a}\leq\frac{1}{3}<\frac{1}{2}$. Moreover only diagonal blocks contribute\footnote{In fact if $P_i$ are orthogonal projections s.t. $\sum_i P_i=1$, then for a fibrewise operator $A$ we have  $\Str A=\tr \eta A=\tr (\sum_i P_i\eta AP_i)+\tr(\sum_{i\neq j}P_i\eta A P_j)=\tr (\sum_i P_i\eta AP_i)$.} to the $\Str_\G$,
therefore we only have to guarantee the integrability of $\Str_\G(t^{-\frac{1}{2}}P_\ep DP_\ep e^{-P_\ep\mathbb{B}^2_t P_\ep})$, because from \cite[Prop.11]{HL} $\StrG e^{-T}=\mathcal O(t^{-\de})$, $\,\forall \de>0$.
 
We reason as follows:
$
\Str{}_\G(t^{-\frac{1}{2}}P_\ep D P_\ep e^{-P_\ep\mathbb{B}^2_t P_\ep})=t^{-\frac{1}{2}}\tr{}_\G (UP_\ep)
$, 
where $U=\tau P_\ep DP_\ep e^{-P_\ep\mathbb{B}^2_t P_\ep}$, and $\tau $ is the chirality grading.
 
Next we evaluate $\tr_\G(UP_\ep)=\tr_\G(UP_\ep^2)=\tr_\G(P_\ep UP_\ep)$.
To do this, since our trace has values differential forms, let $\o_1,\dots,\o_J$ a base of $\Lambda T^*_zB$, for $z$ fixed on  $K$. $U$ is a family of operators and $U_z$ acts on $\mathcal{C}^\infty(\Mt_z,\Et_z)\otimes \Lambda T^*_zB$. Write $U_z=\sum_j U_j\otimes \o_j$. 
$$
\tr{}_\G(P_\ep UP_\ep)=\sum_j \tr{}_\G(P_\ep U_jP_\ep)\otimes \o_j=\sum_j \tr (\chi_{\mathcal{F}}P_\ep U_j P_\ep \chi_{\mathcal{F}})\otimes \o_j.
$$
Now $\;\tr(\chi_{\mathcal{F}}P_\ep U_j P_\ep \chi_{\mathcal{F}})=\sum_i<\chi_{\mathcal{F}}P_\ep U_j P_\ep \chi_{\mathcal{F}}\de_{v_i},\de_{v_i}>=\sum_i<U_j P_\ep \chi_{\mathcal{F}}\de_{v_i}, P_\ep \chi_{\mathcal{F}}\de_{v_i}>
$,  where $\{\de_{v_i}\}$ is a base of $L^2(\Mt_z{}_{|\mathcal{F}},\Et_z{}_{|\mathcal{F}})$. Therefore
$$
|<U_j P_\ep \chi_{\mathcal{F}}\de_{v_i}, P_\ep \chi_{\mathcal{F}}\de_{v_i}>|\leq \n{U_j P_\ep \chi_{\mathcal{F}}\de_{v_i}}\cdot \n{P_\ep \chi_{\mathcal{F}}\de_{v_i}}\leq
$$
$$
\leq \n{U_j}\n{P_\ep \chi_{\mathcal{F}}\de_{v_i}}^2\leq \n{U_z}\n{P_\ep \chi_{\mathcal{F}}\de_{v_i}}^2.
$$ Now
$
\sum_i \n{P_\ep \chi_{\mathcal{F}}\de_{v_i}}=
\sum_i<P_\ep \chi_{\mathcal{F}}\de_{v_i},P_\ep \chi_{\mathcal{F}}\de_{v_i}>= \sum_i<\chi_{\mathcal{F}}P_\ep \chi_{\mathcal{F}}\de_{v_i},\de_{v_i}>=\tr{}_\G(P_\ep)=\mathcal{O}(\ep^\beta)
$
where $\b=\a_K$. Hence
$$
\tr{}_{\G}(P_\ep UP_\ep)\leq \n{U}\mathcal{O}(\ep^\b)=\n{U}\mathcal{O}(t^{-\frac{\b}{a}})\;\; ,\;\;\text{with}\;\; \ep=t^{-\frac{1}{a}}
$$
\textbf{Claim} (\cite[Lemma 13]{HL})\textbf{:}   \emph{$\n{t^{-\frac{q}{2}}U}$ is bounded independently of $t$, for $t$ large.}
This follows because $(P_\ep\mathbb B P_\ep)^2= P_\ep D^2 P_\ep-\bar{C}_t$, with $\bar{C}_t$ is a fibrewise differential operator of order at most one with uniformly bounded coefficients. Therefore $\n{t^{-\frac{1}{2}}\bar{C}_t}_{l,l-1}$ is bounded independently of $t$, for $t$ large. Now writing the Volterra series for $e^{-t(P_\ep D^2 P_\ep)^2+\bar{C}_t}$, we have $U=\tau P_\ep \sum_k \int_{\Delta_k} e^{-t\s_0P_\ep D^2P_\ep}\bar{C}_te^{-t\s_1P_\ep D^2P_\ep}\dots \bar{C}_te^{-t\s_kP_\ep D^2P_\ep }d\underline{\s}$, then estimating each addend as
\begin{multline*} 
\n{e^{-t\s_0P_\ep D^2P_\ep}\bar{C}_te^{-t\s_1P_\ep D^2P_\ep}}_{l,l}\leq \\
\leq \n{\tau P_\ep D e^{-t \s_0P_\ep D^2P_\ep} }_{l,l+1}\n{\bar{C}_t}_{l+1,l} \n{e^{-t\s_1P_\ep D^2P_\ep}}_{l,l+1}\cdot\dots\cdot \n{\bar{C}_t}_{l+1,l}\n{e^{-t\s_kP_\ep D^2P_\ep}}_{l,l+1}
\end{multline*}
we get the Claim. 

Thus $t^{-\frac{1}{2}}\tr_\G(UP_\ep)\leq c\n{U}t^{-\frac{\b}{a}-\frac{1}{2}}$, and 
$
\Str{}_{\G}(\frac{d\mathbb{B}_t}{dt}e^{-\mathbb{B}_{t,\ep}^2})\leq c t^{\frac{q}{2}-\frac{\b}{a}-\frac{1}{2}}$.
We require then 
$ \frac{q-1}{2}-\frac{\b}{a}<-1$
to have integrability hence we need finally $a<\frac{2\b}{q+1}$. Because $a$ was also required to be $a>6$ (see lines after Lemma \ref{est-rest}), the hypothesis 
\begin{equation}
\label{n.s.condition}
\b>3(q+1)\end{equation}
 is a sufficient condition to have the first term in \eqref{2term} equal $\mathcal O(t^{-1-\de})$, with $\de>0$.

\subsubsection{The term II}
Now let's consider the second term in \eqref{2term}. 
As in \cite[pag.197-198]{HL}, write $\mathbb{B}_t=\sqrt{t}D+\mathbb{B}_1+\frac{1}{\sqrt{t}}\mathbb{B}_2$, and locally $\mathbb{B}_1=d+\Phi$. We have $\ds\frac{d\mathbb{B}_t^2(z)}{dz}=\mathbb{B}_t(z)A_{\ep,t}+A_{\ep,t}\mathbb{B}_t(z)= 
\sqrt{t}D A_1+A_2\sqrt{t}D+A_3$,
 where $A_i=C_{i,1}P_\ep C_{i,2}$,  and $C_{i,j}\in \mathcal{M}_1$ are sums of words in $\Phi$, $d(\Phi)$, $t^{-\frac{1}{2}}\mathbb{B}_{[2]}$, $t^{-\frac{1}{2}}d(\mathbb{B}_{[2]})$. 
This implies that $C_{i,j}$ are differential operators with coefficients uniformly bounded in $t$. 
\begin{multline*}
\Str{}_\G\left(\frac{d\mathbb{B}_t}{dt}F_{\ep,t}\right)=\tr{}_\G \tau(t^{-\frac{1}{2}}D-t^{-\frac{3}{2}}\mathbb{B}_{[2]})\int_0^1\int_0^1 e^{-(s-1)\mathbb{B}^2_t(z)}(\sqrt{t}DC_{1,1}P_\ep C_{1,2}+\\
+
C_{2,1}P_\ep C_{2,2}\sqrt{t}D+C_{3,1}P_\ep C_{3,2})e^{-s\mathbb{B}^2_t(z)}dsdz=\\
=\tr{}_\G \int_0^1\int^1_0 \left[C_{1,2}e^{-s\mathbb{B}^2_t(z)}\tau \left(\frac{D}{\sqrt t}-\frac{\mathbb{B}_{[2]}}{\sqrt t^{3}}\right) e^{-(s-1)\mathbb{B}^2_t(z)}\sqrt t D C_{1,1}P_\ep\right.+\\
+C_{2,2}\sqrt t D e^{-s\mathbb{B}^2_t(z)}\tau \left(\frac{D}{\sqrt t}-\frac{\mathbb{B}_{[2]}}{\sqrt t^{3}}\right) e^{-(s-1)\mathbb{B}^2_t(z)}C_{2,1}P_\ep+\\
\left.+ C_{3,2}e^{-s\mathbb{B}^2_t(z)}\tau \left(\frac{D}{\sqrt t}-\frac{\mathbb{B}_{[2]}}{\sqrt t^{3}}\right) e^{-(s-1)\mathbb{B}^2_t(z)}
C_{3,1}P_\ep\right] ds dz= \tr{}_\G (P_\ep WP_\ep)
\end{multline*}
with $W$ the term in square brackets.

With a similar argument as in the Claim above and as in \cite[p. 199]{HL}, we have that $\n{t^{-\frac{q}{2}}e^{-s\mathbb{B}^2_t(z)}\tau e^{-(s-1)\mathbb{B}^2_t(z)} }$ is bounded independently of $t$ as $t\rightarrow \infty$ so that the condition \eqref{n.s.condition} on the Novikov--Shubin exponent guaranties that the term \emph{II.} is $\mathcal O(t^{-1-\de})$ as $t\rightarrow\infty$ as well.
\end{proof}
Theorem \ref{th-eta} and Proposition \ref{intloc} taken together imply
\begin{cor}\label{coreta}
If $\mathcal{\tilde{D}}$ has regular spectral projections and regularity $>3(\dim B+1)$
\begin{equation*}
\hat{\eta}_{(2)}(\mathcal{\tilde{D}})=\left\{\left.\begin{array}{c}\ds\int_0^\infty \Str_\G \ds \left(\frac{d\mathbb{\tilde{B}}_t}{dt}e^{-\mathbb{\tilde{B}}_t^2}\right) dt\;\; \text{if}\; \dim \tilde{Z}=\text{ even}
  \\\ds \int_0^\infty \tr^{even}_\G \ds \left(\frac{d\mathbb{\tilde{B}}_t}{dt}e^{-\mathbb{\tilde{B}}_t^2}\right) dt\;\; \text{if}\; \dim \tilde{Z}=\text{ odd}
\end{array}\right.\right.
\end{equation*}
is well defined as a continuos differential form on $B$.
\end{cor}

\begin{remark} \label{conv+} Theorem \ref{th-eta} gives $\hat{\eta}_{(2)}$ as a continuos form on $B$. Therefore $\hat{\eta}_{(2)}$ fits into a weak $L^2$-local index theorem (see \cite{GR, AGS}).  To get a strong local index theorem one should prove estimates for $\StrG(\ds\frac{d\mathbb{B}_t}{dt}e^{-\mathbb{B}_t^2})$ in $C^1$-norm, assuming more regularity on $\a_K$. 
\end{remark}

\begin{remark}
If $Z$ odd dimensional,  $\hat{\rho}_{(2)}$ is an even degree differential form, whose zero degree term is a continuos function on $B$ with values the Cheeger--Gromov $L^2$-eta invariant of the fibre, $\hat{\eta}_{(2)}^{[0]}(b)=\eta_{(2)}(D_b, \Mt_b\rightarrow M_b)$.
\end{remark}

\subsection{Case of uniform invertibility}\label{unifinv}
Suppose the two families $\mathcal D$ and $\tilde{\mathcal D}$ are both uniformly invertible, i.e. 
\begin{equation}\label{inv} 
\exists \mu>0 \text{ such that } \;\forall b\in B \;\;\; \left\{\begin{array}{c}\spec(D_b)\cap (-\mu,\mu)=\emptyset \\ 
\spec(\tilde{D}_b)\cap (-\mu,\mu)=\emptyset
\end{array}\right.
\end{equation}
In this case the $t\rightarrow\infty$ asymptotic is easy and in particular $\ds\Str_\G(\frac{d\mathbb{B}_t}{dt}e^{-\mathbb{B}_t^2})=\mathcal O(t^{-\de})$, $\forall \de>0$ \cite{AzT}. With the same estimates (see \cite[p. 194]{HL}) one can look at $\ds\frac{\partial}{\partial b}\Str_\G(\frac{d\mathbb{B}_t}{dt}e^{-\mathbb{B}_t^2})$ and obtain that 
$\Str_\G(\frac{d\mathbb{B}_t}{dt}e^{-\mathbb{B}_t^2})\stackrel{C^1}{=}\mathcal O(t^{-\de})$,  $\forall \de>0$.

\section{The $L^2$ rho form}

\begin{definition}
\label{rho2}
Let $(\pi\colon M\rightarrow B, g^{M/B}, \mathcal V, E)$ be a geometric family, $p\colon \Mt\rightarrow M$ a normal covering of it. Assume that $\ker \mathcal{D}$ forms a vector bundle, and that the family $\mathcal{\tilde{D}}$ has regular projections with family Novikov--Shubin invariants $\a_K>3(\dim B+1)$. We define the \emph{$L^2$-rho form} to be the difference
$$
\hat{\rho}_{(2)}(M,\Mt,\mathcal{D}):=\hat{\eta}_{(2)}(\tilde{\mathcal{D}})
-\hat{\eta}(\mathcal{D})\;\;\in \mathcal{C}^0(B,\Lambda T^*B).
$$
\end{definition}
\begin{remark}
When the fibres are odd dimensional,  $\hat{\rho}_{(2)}$ is an even degree differential form, whose zero degree term is a continuos function on $B$ with values the Cheeger--Gromov $L^2$-rho invariant of the fibre, $\hat{\rho}_{(2)}^{[0]}(b)=\rho_{(2)}(D_b,\Mt_b\rightarrow M_b)$.
\end{remark}
We say a continuos $k$-form $\phi$ on $B$ \emph{has weak exterior derivative $\psi$} (a $(k+1)$-form) if, for each smooth chain $c\colon \Delta_{k+1}\rightarrow B$, it holds $\ds\int_c\psi=\int_{\partial c}\phi$, and we write $d\phi=\psi$.
\begin{proposition}\label{wclosed} If $\pi\colon M\rightarrow B$ has odd dimensional fibres,  
 $\hat{\rho}_{(2)}(\mathcal{D})$ is weakly closed.
\end{proposition}
\begin{proof}
 From \eqref{transgs},
$\ds
\int_c\trG{}^{odd}e^{-\tilde{\mathbb{B}}_t^2}-\int_c\trG{}^{odd} e^{-\tilde{\mathbb{B}}^2_T}=\int_{\partial c}\int_t^T\trG{}^{even}\left(\frac{\partial \mathbb{\tilde{B}}_t}{\partial t} e^{-\mathbb{\tilde{B}}_t^2}\right)dt$. 
Taking the limits $t\rightarrow 0$, $T\rightarrow \infty$ we get
\begin{equation*}
\int_c \int_{M/B}\hat{A}(M/B)\ch(E/S)=\int_{\partial c}\hat{\eta}_{(2)}(\tilde{\mathcal{D}})
\end{equation*}
because $
\lim_{T\rightarrow \infty}\tr{}^{odd} e^{-\mathbb{B}_T^2}=\tr(e^{-\nabla_0^2})^{odd}=0
$ because $\tr(e^{-\nabla_0^2})$ is a form of even degree.
The same happens for the family $\tilde{\mathcal{D}}$ where $\ds \int_{M/B}\hat{A}(M/B)\ch(E/S)=d\hat{\eta}(\tilde{\mathcal{D}})\,$ (strongly). Then $\ds\int_{\partial c}\hat{\rho}_{(2)}(\mathcal{D})=0$, which gives the result.
\end{proof}

\begin{cor}
Under uniform invertibility hypothesis \eqref{inv} the form $\hat{\rho}_{(2)}(\mathcal{D})$ is always (strongly) closed.
\end{cor}
\begin{proof}  
The argument is standard: from transgression formul\ae{}  \eqref{transg} \eqref{transgs}, asymptotic behaviour, and Remark \ref{conv+}, we have 
$\ds d\hat{\eta}(\mathcal{D})=\int_{M/B}  \hat{A}(M/B)\ch(E/S)= d \hat{\eta}_{(2)}(\tilde{\mathcal{D}})$.
\end{proof}

\section{$\hat{\rho}_{(2)}$ and positive scalar curvature for spin vertical bundle}
\label{sec:PSC}
Let $\pi\colon M\rightarrow B$ be a smooth fibre bundle with compact base $B$.
If $\hat{g}$ denotes a metric on the vertical tangent bundle $T(M/B)$, and $b\in B$, denote with $\hat{g}_b$ the metric induced on the fibre $M_b$, and write $\hat{g}=(\hat{g}_b)_{b\in B}$. 
Define \begin{equation*}
\mathcal{R}^+(M/B):=\{\hat{g}\;\;\text{metric on } \;T(M/B) \;|\; \scal{}\hat{g}_b>0\;\;\forall b\in B\}
\end{equation*}
to be the space of positive scalar curvature vertical metrics (= PSC).

Assume that $T(M/B)$ is spin and let $\hat{g}\in \mathcal{R}^+(M/B)\neq\emptyset$. By Lichnerowicz formula the family of Dirac operators $\cDi_{\hat{g}}
$ is uniformly invertible. 
Let $p\colon \Mt\rightarrow M$ be a normal $\G$-covering of $\pi$, with $\Mt\rightarrow B$ having connected fibres, and denote with $r\colon M\rightarrow B\G$ the map classifying it. The same holds for $\tilde{\cDi}_{\hat{g}}$, so that we are in the situation of \eqref{unifinv}.

On the space $\mathcal{R}^+(M/B) $ we can define natural relations, following \cite{PS2}.  We say $\hat{g}_0$, $\hat{g}_1\in \mathcal{R}^+(M/B)$ are \textit{path-connected} if there exists a continuos path $\hat{g}_t\in \mathcal{R}^+(M/B) $ between them. 


We say $\hat{g}_0$ and $\hat{g}_1$ are \textit{concordant} if on the bundle of the cylinders $\Pi\colon M\times I\rightarrow B$, $\Pi(m,t)=\pi(m)$, there exists a vertical metric $\hat{G}$ such that: $\forall b\in B$ $\hat{G}_b$ is of product-type near the boundary, $\scal(\hat{G}_b)>0$, and on $M\times\{i\}\rightarrow B$ it coincides with $\hat{g}_i$, $i=0,1$.

\begin{proposition}\label{conc}
Let $\pi\colon M\rightarrow B$ be a smooth fibre bundle with $T(M/B)$ spin and $B$ compact. Let $p\colon\Mt\rightarrow M$ be a normal $\G$-covering of the fibre bundle, such that $\Mt\rightarrow B$ has connected fibres. Then the rho class $[\hat{\rho}_{(2)}(\cDi)]\in H_{dR}^*(B)$ is constant on the concordance classes of $\mathcal{R}^+(M/B)$.
\end{proposition}

\begin{proof}
Let $\hat{g}_0$ and $\hat{g}_1$ be concordant, and $\hat{G}$ the PSC vertical metric on the family of cylinders. 
The family of Dirac operators  $\cDi_{M\times I/B, \hat{G}}$ 
has as boundary the two families $\cDi^0=(D_z, \hat{g}_{0,z})_{z\in B}$ and $\cDi^1=(D_z, \hat{g}_{1,z})_{z\in B}$, both invertible.  Then the Bismut--Cheeger theorem in \cite{BC} can be applied 
$$
0=\int_{M\times I/ B}\hat{A}(M\times I/B) -\frac{1}{2}\hat{\eta}(\cDi_{\hat{g}_0})+\frac{1}{2}\hat{\eta}(\cDi_{\hat{g}_1})\;\;\;\;\; \text{in} \;\;H_{dR}^*(B)
$$
where $\Ch(\mathcal{I}nd \,\mathcal{D}_{M\times I,h})=0\in H_{dR}^*(B)$.

On the family of coverings we reason as before and apply the index theorem in \cite[Theorem 4]{LP} to get 
$$
0=\int_{M\times I/B}\hat{A}(M\times I/B)
-\frac{1}{2}\hat{\eta}_{(2)}(\tilde{\cDi} _{\hat{g}_0})+\frac{1}{2}\hat{\eta}_{(2)}(\tilde{\cDi} _{\hat{g}_1})\;\;\;\;\; \text{in} \;\;\; H^*_{dR}(B)
$$
Subtracting we get $[\hat{\rho}_{(2)}(\cDi_{g_0})]=[\hat{\rho}_{(2)}(\cDi_{g_1})]\in H^*_{dR}(B)$.
\end{proof}

\subsection[$\hat{\rho}_{(2)}$ and the action of a fibre bundle diffeomorphism on $\mathcal R^+(M/B)$]{$\hat{\rho}_{(2)}$ and the action of a fibre bundle diffeomorphism on $\mathcal R^+(M/B)$}

Let $(p, \pi)$ be as in Definition \ref{norcov} and assume further that $p$ is the universal covering of $M$. 

If one wants to use $[\hat{\rho}_{(2)}(\cDi)]$ for the study of $\mathcal R^+(M/B)$ it is important to check how this invariant changes when $\hat{g}\in \mathcal R^+(M/B)$ is acted on by a fibre bundle diffeomorphism $f$ preserving the spin structure. 
\texttt{
\begin{proposition}Let $f\colon M\rightarrow M$ be a fibre bundle diffeomorphism preserving the spin structure. Then $[\hat{\rho}_{(2)}(\cDi_{\hat{g}})]= [\hat{\rho}_{(2)}(\cDi_{f^*\hat{g}})] $
\end{proposition}
}
\begin{proof}
We follow the proof \cite[Prop. 2.10]{PS2} for the Cheeger--Gromov rho invariant. 
Let $\hat{g}$ be a vertical metric and denote $\mathcal S=P\mathit{Spin}(M/B)$ a fixed spin structure, i.e. a 2-fold covering\footnote{or, equivalently, a 2-fold covering of $PGL_+(T(M/B))$ which is not trivial along the fibres of $PGL_+(T(M/B))\rightarrow M$, \cite[p. 8]{PS2}.} of $PSO_{\hat{g}}(T(M/B))\rightarrow M$. 

The eta form downstairs of $\cDi$ depends in fact on $\hat{g}$, on the spin structure, and on  the horizontal connection $T^HM$, so we write here explicitly $\hat{\eta}(\cDi_{\hat{g}})=\hat{\eta}(\cDi_{\hat{g}, \mathcal S}, T^HM)$.

First of all $\hat{\eta}(\cDi_{\hat{g},  \mathcal S}, T^HM)=\hat{\eta}(\cDi_{f^*\hat{g}, f^*\mathcal S}, f^*T^HM)$, because $f$ induces a unitary equivalence between the superconnections constructed with the two geometric structures.

Because $f$ spin structure preserving, 
it induces an isomorphism $\beta_{GL_+}$ between the original spin structure $ \mathcal S$ and the pulled back one $df^* \mathcal S$. Then $\beta_{GL_+}$ gives a unitary equivalence between the operator obtained via the pulled back structures, and the Dirac operator for $f^*\hat{g}$ and the chosen fixed spin structure, so that $\hat{\eta}(\cDi_{f^*\hat{g}, f^* \mathcal S}, f^*T^HM)=\hat{\eta}(\cDi_{f^*\hat{g},  \mathcal S}, f^*T^HM)$. Taken together
\begin{equation*}
\hat{\eta}(\cDi_{\hat{g},  \mathcal S}, T^HM)=\hat{\eta}(\cDi_{f^*\hat{g},  \mathcal S}, f^*T^HM)
\end{equation*}
Let $p\colon \Mt\rightarrow M$ be the universal covering.
Now we look at $\hat{\eta}_{(2)}(\tilde{\cDi})=\hat{\eta}_{(2)}(\tilde{\cDi}_{\hat{g},\mathcal{S}}, T^HM, p)$, where on $\Mt$  the metric, spin structure and connection are the lift via $p$ as by definition. Again, if we construct the $L^2$ eta form for the entirely pulled back structure, we get
$
\hat{\eta}_{(2)}(\tilde{\cDi}_{\hat{g},\mathcal{S}}, T^HM, p)= \hat{\eta}_{(2)}(\tilde{\cDi}_{f^*\hat{g},f^*\mathcal{S}}, f^*T^HM, f^*p)
$. Proceeding as above on the spin structure,
$\hat{\eta}_{(2)}(\tilde{\cDi}_{f^*\hat{g},f^*\mathcal{S}}, f^*T^HM, f^*p)=\hat{\eta}_{(2)}(\tilde{\cDi}_{f^*\hat{g},\mathcal{S}}, f^*T^HM, f^*p)$. 
Since $\Mt$ is the universal covering we have a covering isomorphism between $f^*\Mt$ and $\Mt$, which becomes an isometry when $\Mt$ is endowed of the lift of the pulled back metric $f^*\hat{g}$, therefore 
$$\hat{\eta}_{(2)}(\tilde{\cDi}_{f^*\hat{g},\mathcal{S}}, f^*T^HM, f^*p)=\hat{\eta}_{(2)}(\tilde{\cDi}_{f^*\hat{g},\mathcal{S}}, f^*T^HM, p)$$
It remains to observe how $\hat{\eta}$ and $\hat{\eta}_{(2)}$ depends on the connection $T^HM$. We remove for the moment the hat $\hat{\,}$ to simplify the notation.
Let $T^H_0M,T^H_1M$ two connections, say given by $\o_0,\o_1\in \O^1(M, T(M/B))$ and pose $\o_t=(1-t)\o_0+t\o_1$. Construct
the family $\breve{M}=M\times [0,1]\stackrel{\breve{\pi}}{\rightarrow} B\times [0,1]=:\breve{B}$ as in the proof of Prop. \ref{intloc}. On this fibre bundle put the connection one form $\breve \o+dt$. Since $\breve d \breve {\eta}=d\breve{\eta}(\cdot , t)-\frac{\partial}{\partial t} \eta(t)dt$ we have
$$
\eta_0-\eta_1=\int_0^1 \breve d \breve{\eta}-\int_0^1 d i_{\frac{\partial}{\partial t}}\breve{\eta}=\int_0^1\int_{\breve{M}/\breve{B}}\hat{A}(M\times I/B\times I)- d\,\int_0^1 i_{\frac{\partial}{\partial t}}\breve{\eta}
$$
which is the sum of a local contribution plus an exact form. 
Writing the same for $\eta_{(2)}$ we get that for the  $L^2$-rho form
$\hat{\rho}_{(2)}(\cDi,T^H_0M)=\hat{\rho}_{(2)}(\cDi,T^H_1M)\in \O(B)/d\O(B)$ and therefore we get the result.
\end{proof}

\subsection{Conjectures}
Along the lines of \cite{Ke,PS1} we can state the following conjectures.
\begin{conjecture}
If $\G$ is torsion-free and satisfies the Baum-Connes conjecture for the maximal $C^*$-algebra, then $[\hat{\rho}_{(2)}(\cDi_{\hat{g}})]$ vanishes if $\hat{g}\in \mathcal{R}^+(M/B)$.
\end{conjecture}
\begin{definition} Let $\pi\colon M\rightarrow B$ and $\theta\colon N\rightarrow B$ be two smooth fibre bundles of compact manifolds over the same base $B$. A continuos map $h\colon N\rightarrow M$ is called a \emph{fibrewise homotopy equivalence} if  $\pi\circ h=\theta$, and there exists $g\colon N\rightarrow M$ such that $\theta\circ g=\pi$ and such that $h\circ g$, $g\circ h$ are homotopic to the identity by homotopies that take each fibre into itself.
\end{definition}

We work in the following with smooth fibrewise homotopy equivalences.

\begin{definition}Let $\G$ be a discrete group and $(\pi\colon M\rightarrow M, p\colon \Mt\rightarrow M)$, $(\theta\colon N\rightarrow B, q\colon \Nt\rightarrow N)$ be two  normal $\G$-coverings of the fibre bundles $\pi$ and $\theta$. Denote as $r\colon M\rightarrow B\G$, $s\colon N\rightarrow B\G$ the two classifying maps. 
We say $(\pi,p)$ and $(\theta, q)$ are \emph{$\G$-fibrewise  homotopy equivalent} if there exists  a fibrewise homotopy equivalence $h\colon N\rightarrow M$ such that $s\circ h$ is homotopic to $r$.
\end{definition}
Let $\mathcal{D}^{sign}$ denote the family of signature operators.
\begin{conjecture}
Assume $\G$ is a torsion-free group that satisfies the Baum-Connes conjecture for the maximal $C^*$-algebra.
Let $h$ be a orientation preserving $\G$-fibrewise  homotopy equivalence between $(\pi,p)$ and $(\theta, q)$ and suppose $\mathcal{\tilde{D}}^{sign}_{M/B}$ and $\mathcal{\tilde{D}}^{sign}_{N/B}$ have smooth spectral projections and Novikov--Shubin invariants $>3(\dim B+1)$. 

Then $[\hat{\rho}_{(2)}(\mathcal{\tilde{D}}^{sign}_{M/B})]=[\hat{\rho}_{(2)}(\mathcal{\tilde{D}}^{sign}_{N/B})]\in H^*_{dR}(B)$.
\end{conjecture}


\appendix

\section{Analysis on normal coverings}
\label{app1}
We summarize the analytic tools we use to investigate $L^2$ spectral invariants, namely $\mathcal{N}\G$-Hilbert spaces and Sobolev spaces on manifolds of bounded geometry, following the nice exposition in \cite{Va}.

\subsection{$\mathcal{N}\G$-Hilbert spaces and von Neumann dimension} 
Let $\G$ be a discrete countable group and $l^2(\G)$ the Hilbert space of complex valued, square integrable functions on $\G$. Denote with $\de_\g\in \C\G$ the function with value $1$ on $\g$, and zero elsewhere. The convolution law on $\C\G$ is $\de_\gamma*\de_\beta=\de_{\gamma\b}$.
Let $L$ be the action of $\G$ on $l^2(\G)$ by left convolution $L\colon \G\rightarrow \mathcal{U}(l^2(\G))$, $L_\g(f)=(\de_\g*f )(x)=f(\g^{-1}x)$. 
Right convolution action is denoted by $R$.

\begin{definition}
The \emph{group von Neumann algebra} $\mathcal{N}\G$ is defined to be the weak closure $\mathcal{N}\G:=L\overline{(\mathbb{C}\G)}^{weak}$ in $ \mathcal{B}(l^2(\G))$.
By the double commutant theorem
$\mathcal{N}\G =R(\mathbb{C}\G)'$, so that $\mathcal{N}\G$ is the algebra of operators commuting with the right action of $\G$.
An important feature of the group von Neumann algebra is its \emph{standard trace} $ 
\trG\colon \mathcal{N}\G\longrightarrow \mathbb{C}$ defined as $\trG A= <A\de_e,\de_e>_{l^2(\G)}$. In particular for $A=\sum a_\g L_\g\in \mathcal{N}\G$, then $\trG(A)=a_e$.\end{definition}

\begin{definition}\label{free}
A \emph{free $\mathcal{N}\G$-Hilbert space} is a Hilbert space of the form $W\otimes l^2(\G)$, where $W$ is a Hilbert space and $\G$ acts on $l^2(\G)$ on the right.\\
A \emph{$\mathcal{N}\G$-Hilbert space} $\mathcal{H}$ is a Hilbert space with a unitary right-action of $\G$ such that there exists a $\G$-equivariant immersion
$
H\rightarrow \mathcal{V}\otimes l^2(\G)
$
in some free $\mathcal{N}\G$-Hilbert space.
For $\mathcal{H}_1,\mathcal{H}_2$ $\mathcal{N}\G$-Hilbert spaces, define
$\mathcal{B}_\G(\mathcal{H}_1\mathcal{H}_2)\colon =\{T\colon \mathcal{H}_1\rightarrow \mathcal{H}_2\;\;\text{bounded}\; \text{and } \;\G\text{-equivariant}\}$.
\end{definition}
Let $\mathcal{H}=\mathcal{V}\otimes l^2(\G)$ be a free $\mathcal{N}\G$-Hilbert space. Then $\mathcal{B}_\G(\mathcal{V}\otimes l^2(\G))\simeq \mathcal{B}(\mathcal{V})\otimes \mathcal{N}\G$.
There exist a trace on the positive elements of this von Neumann algebra, with values in $[0,\infty]$: let $(\psi_j)_{j\in \mathbb{N}}$ is a orthonormal base of $\mathcal{V}$;
 if $f\in \mathcal{B}(\mathcal{H})_+$, its trace is given by $
\trG(f)=\sum_{j\in \mathbb{N}}<f(\psi_j\otimes \de_e),\psi_j\otimes \de_e>$. A $\G$-trace can be defined also on any  $\mathcal{N}\G$-Hilbert-space $H$ using the immersion $ j\colon H\hookrightarrow \mathcal{V}\otimes  l^2\G$ and proveing that the trace does not depend on the choice of $j$ (see \cite{CG} or \cite[pag. 17]{Lu}).

\begin{definition}Let $\mathcal{H}$ be a $\mathcal{N}\G$-Hilbert space. Its \emph{von Neumann dimension} is defined as
$\dim{}_{\G}(\mathcal{H})=\trG (\id\colon \mathcal{H}\rightarrow \mathcal{H})\;\;\in\;[0,+\infty) $.
\end{definition}

\begin{definition} Let $\mathcal{H}_1$ and $\mathcal{H}_2$ be $\mathcal{N}\G$-Hilbert spaces. Define 
\begin{itemize}
  \item $
\mathcal{B}^f_\G(\mathcal{H}_1,\mathcal{H}_2):=\{A\in \mathcal{B}_\G( \mathcal{H}_1,\mathcal{H}_2)' | \dim{}_\G \left(\overline{\Im A}\right)<\infty\}$ are the \emph{$\G$-finite rank operators} 
  \item $
\mathcal{B}^\infty_\G(\mathcal{H}_1,\mathcal{H}_2):=\overline{\mathcal{B}^f_\G(\mathcal{H}_1,\mathcal{H}_2)}^{\n{\;\;}}$, are the \emph{$\G$-compact operators} 
  \item $
\mathcal{B}^2_\G(\mathcal{H}):=\{A\in \mathcal{B}_\G( \mathcal{H}) s.t. \trG \left(AA^*\right)<\infty\}$, are the \emph{$\G$-Hilbert-Schmidt operators} 
  \item $\mathcal{B}^1_\G(\mathcal{H}):=\mathcal{B}^2_\G(\mathcal{H})\mathcal{B}^2_\G(\mathcal{H})^*$ the \emph{$\G$-trace class} operators.
\end{itemize}
\end{definition}
Their main properties are:\begin{itemize}
\item[1)] $ \mathcal{B}^f(\mathcal{H}),  \mathcal{B}^\infty(\mathcal{H}),  \mathcal{B}^2(\mathcal{H}),  \mathcal{B}^1(\mathcal{H})$ are ideals and $\mathcal{B}^f\subset \mathcal{B}^1\subset \mathcal{B}^2\subset \mathcal{B}^\infty$;
\item[2)] $A\in \mathcal{B}^i(\mathcal{H})$ if and only if $|A|\in \mathcal{B}^i(\mathcal{H})$ for $i=1,2,f,\infty$.
\end{itemize}

\subsection{Covering spaces, bounded geometry techniques}
Let $p\colon \Zt\rightarrow Z$ a normal $\G$-covering of a compact Riemannian manifold $Z$. Let  $\mathcal{I}\subset \Zt $ be a fundamental domain for the (right) action of $\G$ on $\Zt $ ($\mathcal{I}$ is an open subset s.t. $\ds\mathcal{I}\cdot\g\cap \mathcal{I}$ and $\Zt\setminus \bigcup  \mathcal{I} \cdot\g$ have zero measure $\forall \g\neq e$).

Let $E\rightarrow Z$ a Hermitian vector bundle, and $\Et=p^* E$ the pull-back. The sections $\mathcal{C}^\infty_c(\Zt,\Et)$ form a $\C\G$-right module for the action 
$\displaystyle 
(\xi\cdot f) (\tilde{m})=\sum_{g\in \G}(R^*_g\xi)(\tilde{m}) f(g^{-1})$
where $(R^*_g\xi)(\tilde{m}):=\xi(\tilde{m}g)$.
Its Hilbert space completion $L^2(\Zt,\Et)$ is a $\G$-free Hilbert space in the sense of definition \ref{free}, in fact the map
$\psi\colon L^2(\Zt,\Et)\longrightarrow L^2(\mathcal{I},\Et_{|\mathcal{I}})\otimes l^2(\G)$, $\xi \mapsto \sum_{\g\in\G} \left(R_\g^*\xi\right)_{|\mathcal{I}}\otimes \delta_{\g}$
is an isomorphism.

The $\G$-trace class operators are characterized as follows: let $A\in \mathcal{B}_\G(L^2(\Zt,\Et))$
\begin{equation*}
  A\in \mathcal{B}^1_\G(L^2(\Zt,\Et))\text{ if and only if }\chi_{\mathcal{I}}|A|\chi_{\mathcal{I}}\in \mathcal{B}^1(L^2(\mathcal{I},E_{|\mathcal{I}}))
  \end{equation*}
If $A\in \mathcal{B}^1_\G(L^2(\Zt,\Et))$ then $ \tr_\G(A)=\tr(\chi_{\mathcal{I}}A \chi_{\mathcal{I}})$. If $A\in \mathcal{B}^1_\G(L^2(\Zt,\Et))$ has
 Schwartz kernel $[A]$ continuos, then \begin{equation}
\label{trace}
\tr A=\int_{\mathcal{I}}\tr{}_{\Et_x}\left([A](x,x)\right)dx=\int_M \pi_*\tr{}_{\Et_x}\left([A](x,x)\right)dx\,.
\end{equation}
The covering of a compact manifold and the pulled back bundle $\tilde{E}$ above are the most simple examples of manifolds of  \emph{bounded geometry}\footnote{Let $(N,g)$ be a Riemannian manifold.
$N$ is \textit{of bounded geometry} if 
\begin{enumerate}
\item it has positive injectivity radius $i(N,g)$;
\item the curvature $R_N$ and all its covariant derivatives are bounded.
\end{enumerate}
A hermitian vector bundle $E\rightarrow N$ is \textit{of bounded geometry} if the curvature $R^E$ and all its covariant derivatives are bounded. This can be characterized in normal coordinates with conditions on $g$, coordinate transformations and $\nabla$ (see for example in \cite{Ro} and \cite{Va}).
}.

The analysis on manifolds of bounded geometry was developped in \cite{Ro}. We specialize here to the case of a normal covering $\tilde{Z}$. 

The Sobolev spaces of sections are defined, for $k\geq 0$, as the completion $H^k(\Zt,\Et):=\overline{C_c^\infty(\Zt,\Et)}^{\n{\;}_k}$ where $\n{f}_k:=\sum_{j=0}^k\n{\nabla^j f}_{L^2(\Zt,\Et \otimes ^j T^* \Zt)}$; for $k<0$ $H^k(\Zt,\Et)$ is defined as the dual of $H^{-k}(\Zt,\Et)$.
%

The spaces of \emph{uniform $\mathcal{C}^k$ sections} are defined as follows: $U\mathcal{C}^k(\Mt)=\{f\colon \Mt\rightarrow \mathbb{C}\; |\, f\in \mathcal{C}^k\; \text{and} \; \n{f}_k\leq c(k)\; \forall k\}  $, where $\n{f}_k=\sup_{\tilde{m}\in \Mt, X_i} \{|\nabla_{X_1}\dots \nabla_{X_k}f (\tilde{m})|\}$, and analogously for sections $U\mathcal{C}^k(\Mt,\Et)$. 
$U\mathcal{C}^\infty(\Mt,\Et)$ is the Fr\'echet space $:=\bigcap_k U\mathcal{C}^k(\Mt,\Et)$.

The following \emph{Sobolev embedding} property holds \cite{Ro}: if $\dim \Mt=n$, then for $j>\frac{n}{2}+k$ there is a continuos inclusion $H^j(\Mt,\Et)\hookrightarrow U\mathcal{C}^k(\Mt,\Et)$.

The algebra $\UDiff (\Mt,\Et)$ of \emph{uniform differential operators} is the algebra generated by operators in $U\mathcal{C}^\infty(\Mt,\End \Et)$ and derivatives $\{\nabla^{\Et}_X\}_{X\in U\mathcal{C}^\infty(\Mt,T\Mt)}$ with respect to uniform vector fields. 
$P\in \UDiff (\Mt,\Et)$  extends to a continuos operator $H^j(\Mt,\Et)\rightarrow H^{j-k}(\Mt,\Et)$  $\forall j\in \mathbb{Z}$.

$P\in \UDiff (\Mt,\Et)$ is called \emph{uniformly elliptic} if its principal symbol $\s_{pr}\in U\mathcal{C}^\infty(T^*\Mt,\pi^*\End \Et)$ is invertible out of an $\ep$-neighborhood of $0\in T^*\Mt$, with inverse section which can be uniformly estimated.

For a uniformly elliptic operator $T$ the \emph{G\aa rding inequality} holds:\begin{equation}
\label{Gaa}
\n{\phi}_{H^{s+k}(\Mt,\Et)}\leq c(s,k) \left( \n{\phi}_{H^s}+\n{T\phi}_{H^s}\right)\;\;\; \forall s\in \mathbb{R}
\end{equation}
If $T$ is a continuos operator $T\colon \mathcal{C}^\infty _c(N,E)\rightarrow (\mathcal{C}^\infty _c(\Mt,\Et))'$ we will denote its \emph{Schwartz kernel} with $[T]\in \mathcal{C}^\infty(\Mt\times \Mt, \Et\XBox \Et^*)$. 
\begin{definition}
We say that $T\colon \mathcal{C}^\infty _c(N,E)\rightarrow (\mathcal{C}^\infty _c(\Mt,\Et))'$ \emph{has order} $k\in \mathbb{Z}$ if $\forall s\in \mathbb{Z}$ it admits a bounded extension
$H^s(\Mt,\Et)\rightarrow H^{s-k}(\Mt,\Et)$. Hence it is closable as unbounded operator on $L^2(\Mt,\Et)$. 

The space of order $k$ operators is denoted $\Op^k(\Mt,\Et)$, and comes with the seminorms on $\mathcal{B}(H^s(\Mt,\Et), H^{s-k}(\Mt,\Et))$. The space $ \Op^{-\infty}(\Mt,\Et)=\bigcap_{k}  \Op^k(\Mt,\Et)$ is a Fr\'echet space. 

Finally, an operator $T\in \Op^k(\Mt,\Et)$ is called \emph{elliptic} if it satisfies G\aa rding inequality. We will denote as $\Op^k_\G(\Mt,\Et)$ the subspace of \emph{$\G$-invariant operators} in $\Op^k(\Mt,\Et)$.
\end{definition}

Consider the Fr\'echet space of continuos rapidly decreasing functions
$$
RB(\mathbb{R})=\{f\colon  \mathbb{R}\rightarrow \mathbb{C}\; :\;f \text{ continuos, and } \left|(1+x^2)^{\frac{k}{2}}f(x) \right|< \infty \;\;\forall k\}
$$
Let $T\in \Op^k(\Mt,\Et)\,, k\geq 1$ an elliptic, formally self-adjoint operator. Denote again by $T$ its closure, with domain $\dom T=H^k(\Mt,\Et)$.
From G\aa rding inequality \eqref{Gaa} the map
$RC(\mathbb{R})\longrightarrow \mathcal{B}(H^j(\Mt,\Et),H^l(\Mt,\Et))$, $ f \mapsto f(T)
$
is continuos $\forall j, l\in \mathbb{Z}$, so that
\begin{equation*}
RC(\mathbb{R})\longrightarrow \Op{}^{-\infty}(\Zt, \Et) \; , f\mapsto f(T)
\end{equation*} 
is continuos. One can prove that the Schwartz kernel of such operator is smooth: by \eqref{Gaa} and Sobolev embedding, for $L=[\frac{n}{2}+1]$ the map 
\begin{equation}
\label{C0sob}
\Op{}^{-2L-l}(\Zt,\Et)\longrightarrow UC^l(\Zt\times \Mt,E\XBox E^*)\,,\,T \mapsto [T]\end{equation} 
is continuos  $\forall l\in \mathbb{N}$; then
in particular for
 $f\in RC(\mathbb{R})$, the kernel $[f(T)]\in U\mathcal{C}^\infty (\Zt\times \Zt,\Et\XBox \Et^*)$ and the map $ RB(\mathbb{R})\longrightarrow U\mathcal{C}^\infty (\Zt\times \Zt,\Et\XBox \Et^*)$ \,,\,$ f\mapsto [f(T)]$ is continuos.

\begin{lemma}\label{appker}
Since elements in $\Op_\G^{-\infty}(\Mt,\Et)$ are $\G$-trace class, then
 for $T\in \Op{}_\G^k(\Mt,\Et)$ elliptic and selfadjoint, the map 
$ RC(\mathbb{R})\rightarrow \mathcal{B}^1_\G(\Mt,\Et)$ ,\,$ f\mapsto f(T)$ is continuos. As a consequence
 $\forall m$ $\exists l$ such that 
 \begin{equation}
\label{ }
 |\trG f(T)|_m\leq C\n{f(T)}_{-l,l}\,.
\end{equation}
\end{lemma}

 \end{document}